\documentclass[reqno]{amsart}%
\usepackage{amssymb}

\usepackage{cite}
\usepackage{hyperref}
\usepackage{index}
\newindex{default}{idx}{ind}{Index}
\newindex{sym}{sdx}{snd}{Symbol index}

\newtheorem{theorem}{Theorem}[part]
\newtheorem{lemma}[theorem]{Lemma}
\newtheorem{proposition}[theorem]{Proposition}
\newtheorem{corollary}[theorem]{Corollary}

\theoremstyle{definition}
\newtheorem{definition}[theorem]{Definition}

\theoremstyle{remark}
\newtheorem{remark}[theorem]{Remark}
\newtheorem{example}[theorem]{Example}

\numberwithin{equation}{part}

\newcommand{\ux}{\underline{x}}
\newcommand{\ov}{\overline}
\newcommand{\Pos}{\rm{Pos}}
\newcommand{\gj}{\mathfrak{j}}
\newcommand{\gm}{\mathfrak{m}}
\newcommand{\gn}{\mathfrak{n}}
\newcommand{\C}{\mathbb{C}}
\newcommand{\R}{\mathbb{R}}
\newcommand{\N}{\mathbb{N}}
\newcommand{\K}{\mathbb{K}}
\newcommand{\cB}{\mathcal{B}}
\newcommand{\cI}{\mathcal{I}}
\newcommand{\cK}{\mathcal{K}}
\newcommand{\cG}{\mathcal{G}}
\newcommand{\cH}{\mathcal{H}}
\newcommand{\cN}{\mathcal{N}}
\newcommand{\cM}{\mathcal{M}}
\newcommand{\cD}{\mathcal{D}}
\newcommand{\cZ}{\mathcal{Z}}
\newcommand{\sA}{\sf{A}}
\newcommand{\ii}{\sf{i}}
\newcommand{\supp}{\rm{supp}}
\newcommand{\sat}{\rm{sat}}

\begin{document}
    \title[Chapter 12: The moment problem on compact \ldots{} (revised version)]{Chapter 12: The moment problem on compact semi-algebraic sets (revised version)}

\author{Konrad Schm\"udgen}
\address{University of Leipzig, Mathematical Institute, Augustusplatz 10/11, D-04109 Leipzig, Germany}
\email{schmuedgen@math.uni-leipzig.de}


\thanks{Acknowledgment: The author would like to thank  Matthias Sch\"otz for the fruitful cooperation.}

\begin{abstract}
    The following is an improved version of Chapter 12 of my book \cite{Sm17}. Among others, we present a new  unified approach to the Archimedean Positivstellens\"atze for quadratic modules and semirings in Section 
    12.4 and we add a number of results on Positivstellens\"atze for semirings and the corresponding moment problems.All references to formulas and to the bibliography of the book are retained.
    
    This version is  essentially  based on results from the recent paper \cite{SmS23}.
    We will also use a result from the  book \cite{Sm20}.
\end{abstract}

\maketitle
\setcounter{part}{12}
\renewcommand{\thesection}{\arabic{part}.\arabic{section}}

In this chapter we begin the study of the multidimensional moment problem. The passage to dimensions $d\geq 2$ brings new difficulties and  unexpected  phenomena. 
In Section 3.2 we  derived
solvability criteria of the moment problem on intervals in terms of positivity conditions. It seems to be natural to look for similar characterizations in higher dimensions as well. This leads us immediately into the  realm of real algebraic geometry and to  descriptions of positive polynomials on semi-algebraic sets. In this chapter we treat this approach   for basic closed {\it compact} semi-algebraic subsets of $\R^d$. It turns out that for such sets there is a  close interaction between  the moment problem  and  real algebraic geometry. Generally speaking, combined with Haviland's theorem any denominator-free Positivstellensatz yields an existence result for the moment problem. We develop this connection in detail and give complete proofs of the corresponding  Positivstellens\"atze.

Basic notions and facts from real algebraic geometry that  are needed   for our treatment of the moment problem are collected in Section \ref{basicssemialgebraicsets}. Section \ref{localizing} contains general facts on localizing functionals and  supports of representing measures.

In Section \ref{momentproblemstrictpos}, we prove our main existence result for the moment problem on compact semi-algebraic sets (Theorem \ref{mpschm}) and the corresponding  Positivstellensatz for preorderings (Theorem  \ref{schmps}).

In Section \ref{reparchimodiules} we derive 
a fundamental result, the  Archimedean Positivstellensatz for quadratic modules and semirings (Theorem \ref{archpos}). In Section \ref{archimedeanpolynomials}, we restate this theorem for the polynomial algebra $\R[x_1,\dotsc,x_d]$ and give applications to the moment problem (Theorems \ref{archmedpsq}, \ref{archmedps}, and \ref{auxsemiring}). Section \ref{polyhedron} contains a Positivstellensatz  and its application to the moment problem (Theorem \ref{prestel}) for semi-algebraic sets which are contained in compact polyhedra. In Section \ref {examplesmp},  we derive a number of classical results and examples on the moment problem for concrete compact sets. 
The results in Sections \ref{momentproblemstrictpos}, \ref{reparchimodiules}, \ref{archimedeanpolynomials}, \ref{polyhedron}, and \ref {examplesmp}  are formulated 
in the language of real algebra, that is, in terms of preorderings, quadratic modules, or  semirings.

Apart from real algebraic geometry  the theory of self-adjoint Hilbert space operators is our  main tool for the multidimensional moment problem.  
In Section \ref{Operator-theoreticappraochtothemomentprblem} we develop this method by studying the GNS construction and the multidimensional spectral theorem. This approach yields a short and elegant approach to the Positivstellensatz and to the moment problem for Archimedean quadratic modules.
\smallskip

Throughout  this chapter, ${\sA}$  denotes a {\bf commutative real algebra with unit element}  denoted by $1$. For notational simplicity we  write $\lambda$ for $\lambda \cdot 1$, where $\lambda \in \R$. Recall that\, $\sum {\sA}^2$\, is  the set of  finite sums $\sum_i a_i^2$ of squares of elements $a_i\in {\sA}$.

\section{Semi-algebraic sets and Positivstellens\"atze}\label{basicssemialgebraicsets}

The following definition contains three basic notions which are needed in the sequel.

\begin{definition}
    A \emph{ quadratic module}\index{Quadratic module} of\, ${\sA}$ is a subset $Q$ of ${\sA}$ such that 
    \begin{align}\label{axiomquadmodule}
        Q + Q \subseteq Q,~~ 1 \in Q,~~
        a^2 Q  \in Q~~{\rm for~all}~~a\in {\sA}.
    \end{align} 
    A quadratic module $T$ is called a {\em preordering}\index{Preordering} if\, $T\cdot T\subseteq T$. \\  A {\em semiring}\index{Semiring}\index{Preprime} is a subset $S$ of ${\sA}$ satisfying
    \begin{align}
        S + S \subseteq S,~~
        S\cdot S  \subseteq S,~~ \lambda \in S ~~{\rm for~all}~~\lambda\in \R,\lambda \geq 0.
    \end{align}
\end{definition}
In the literature ``semirings" are also called ``preprimes".
The name ``quadratic module" stems from the last condition in (\ref{axiomquadmodule}) which means  that $Q$ is invariant under multiplication by squares. Setting $a=\sqrt{\lambda}$, this implies that $\lambda\cdot Q\subseteq Q$ for $\lambda \geq 0$.  While  semirings and preorderings  are closed under multiplication,  quadratic modules are  not necessarily. Semirings  do not contain all squares in general. Clearly, a quadratic module is a preordering if and only if it is a semiring. In this book, we  work mainly with quadratic modules and preorderings. 
%Semirings will  occur only in Theorem \ref{reptheoremaqmodules},   \ref{prestel}, and \ref{mpprestelth} below.

\begin{example}
    The subset\, $S=\{ \sum_{j=0}^n a_jx^j:\, a_j\geq 0, n\in \N\}$\, of\, $\R[x]$\, is a semiring, but not a quadratic module. Clearly,
    $Q=\sum \R_d[\ux]^2+x_1\sum \R_d[\ux]^2+x_2\sum \R_d[\ux]^2$  is a quadratic module of $\R_d[\ux], d\geq 2$, but $Q$ is  neither a semiring nor a preordering. $\hfill \circ$ \end{example}

Obviously, $\sum {\sA}^2$ is the smallest quadratic module of ${\sA}$. Since ${\sA}$ is commutative,   $\sum {\sA}^2$ is invariant under multiplication, so  it is also the smallest  preordering of ${\sA}$. 

Our guiding example for\,  ${\sA}$\, is the polynomial algebra\, $\R_d[\ux]:=\R[x_1,\dotsc,x_d]$. 

Let ${\sf f}=\{f_1,\dotsc,f_k\}$ be a finite subset of $\R_d[\ux]$. The set\index[sym]{KCfB@$\cK({\sf f}), \cK(f_1,\dotsc,f_k)$}
\begin{align}\label{basicclosed}
    \cK({\sf f})\equiv \cK(f_1,\dotsc,f_k)=\{x\in \R^d: f_1(x)\geq 0,\dotsc,f_k(x)\geq 0\}
\end{align}
is called the {\em basic closed semi-algebraic set associated with $\sf f$}.\index{Basic closed semi-algebraic set}\index{Semi-algebraic set} It is easily seen that\index[sym]{QAfB@$Q({\sf f}), Q(f_1,\dotsc,f_k)$}
\begin{align}\label{quadraticqf}
    Q({\sf f})\equiv Q(f_1,\dotsc,f_k)= \big\{\, \sigma_0+ f_1 \sigma_1+\dots+f_k \sigma_k : \,\sigma_0,\dotsc,\sigma_k\in \sum\R_d[\ux]^2\big\}
\end{align}
is the {\em quadratic module generated by the set} $\sf f$,
\begin{align}\label{semiringf}
    S({\sf f}) \equiv S(f_1,\dotsc,f_k)=
    \bigg\{ 
    \sum_{n_1,\dotsc,n_k=0}^r \alpha_{n_1,\dotsc,n_k} f_1^{n_1}\cdots f_r^{n_r}: 
    \alpha_{n_1,\dotsc,n_r}\geq 0, t\in \N_0
    \bigg\}
\end{align}
is the {\it semiring generated by}\, ${\sf f}$, and\index[sym]{TAfB@$T({\sf f}), T(f_1,\dotsc,f_k)$} 
\begin{align}\label{preorderingtf}
    T({\sf f})\equiv T(f_1,\dotsc,f_k)=\bigg\{  \sum_{e=(e_1,\dotsc,e_k)\in \{0,1\}^k} f_1^{e_1}\cdots f_k^{e_k} \sigma_e:\, \sigma_e\in \sum\R_d[\ux]^2 \, \bigg\}
\end{align}
is the {\em preordering generated by the set $\sf f$}.

These sets $\cK({\sf f})$, $Q({\sf f})$, $S({\sf f})$, $T({\sf f})$  play a crucial role in this chapter and the next.
\begin{definition}
    A {\em cone}\index{Cone} is a subset $C$ of ${\sA}$ such that
    \begin{align*}
        C+C\subseteq C~~~ {\rm and}~~~  \lambda \cdot C\subseteq C~~~ {\rm for}~~ \lambda\geq 0.\
    \end{align*}
    A {\em unital cone} of $\sA$ is a cone $C$ which contain the unit element of\, $\sA$.\\
    An {\em $S$-module} for a semiring $S$ is a unital cone such that
    \begin{align}\label{smodules}
        ac\in C~~~ {\rm for} ~~~  a\in S ~~ {\rm and}~~c\in C.
    \end{align}
\end{definition}
Obviously, semirings, quadratic modules, and preorderings are unital cones. 

Setting $c=1$ in (\ref{smodules}) yields $a\in C$ for  $a\in S$. Thus,  $S\subseteq C$ for any $S$-module $C$.\smallskip

Each cone $C$ of ${\sA}$ yields an ordering $\preceq$\index[sym]{smaller@$\preceq$} on  ${\sA}$\,  by defining 
\begin{align*}a \preceq b \quad {\rm 
        if~ and~ only~ if}\quad b-a \in C.
\end{align*}
\begin{example}\label{examplemodule}
    Let $S$ be a semiring of $\sA$ and $g_0:=1, g_1,\dotsc,g_r\in {\sA}$, where $r\in \N$. Then
    \begin{align*}
        C:=g_0 S+ g_1 S+\cdots+g_rS
    \end{align*}
    is the\, {\it $S$-module of\, $\sA$ generated by $g_1,\dotsc,g_r$}.
\end{example}

By the above definitions, all polynomials from $T({\sf f})$ are nonnegative on $\cK({\sf f})$, but in general $T({\sf f})$ does not exhaust the nonnegative polynomials on $\cK({\sf f})$. 

The following {\it Positivstellensatz of Krivine--Stengle}\index{Positivstellensatz!  Krivine--Stengle}
is  a fundamental result of real algebraic geometry. It 
describes nonnegative resp. positive polynomials on $\cK({\sf f})$ in terms of {\it quotients} of elements of the preordering $T({\sf f})$. 
\begin{theorem}\label{krivinestengle}\index{Theorem! Krivine--Stengle}   Let $\cK({\sf f})$ and $T({\sf f})$ be as above and let $g\in \R_d[\ux]$. Then we have:
    \begin{itemize}
        \item[\em (i)]~ {\rm (Positivstellensatz)}~ $g(x)>0$  for all\, $x\in \cK({\sf f})$\, if and only if there exist polynomials $p,q\in T({\sf f})$  such that $pg=1+q$.
        \item[\em (ii)]~ {\rm (Nichtnegativstellensatz)} $g(x)\geq 0$ for all $x\in \cK({\sf f})$ if and only if there exist $p,q\in T({\sf f})$ and $m\in \N$ such that $pg=g^{2m}+q$.
        \item[\em (iii)]~ {\rm (Nullstellensatz)} $g(x)= 0$ for  $x\in \cK({\sf f})$ if and only if $-g^{2n}\in T({\sf f})$ for some $n\in \N$.
        \item[\em (iv)]~ $\cK({\sf f})$ is empty if and only if~ $-1$ belongs to $T({\sf f}).$
    \end{itemize}
\end{theorem}
\begin{proof} See [PD] or [Ms1]. The original papers are [Kv1]  and [Ste1].
      \end{proof}  

All  {\it ``if"} assertions are easily checked and it is not  difficult to show that all four statements are equivalent, see e.g. [Ms1]. Standard proofs  of Theorem \ref{krivinestengle} as given in [PD]  or [Ms1] are based on the Tarski--Seidenberg transfer  principle.
Assertion (i) of Theorem \ref{krivinestengle} will play an  essential role in the proof of Proposition \ref{prearchcom} below.
\smallskip

Now we  turn to algebraic sets.
For a  subset $S$ of $\R_d[\ux],$  the  real zero set of $S$ is\index[sym]{ZCSA@$\cZ(S)$} 
\begin{align}
    \cZ(S)=\{x\in \R^d: f(x)=0 \quad {\rm for~ all}~~f\in S\}.
\end{align}
A subset $V$ of $\R^d$ of the form $\cZ(S)$ is called a {\it real algebraic set}.\index{Real algebraic set} 

Hilbert's basis theorem [CLO, p. 75] implies
that each real algebraic set is of the form $\cZ(S)$ for some {\it finite} set $S=\{h_1,\dotsc,h_m\}$. In particular, each real algebraic set is a basic closed semi-algebraic set, because
$\cK(h_1,\dotsc,h_m,-h_1,\dotsc,-h_m)=\cZ(S)$.

Let $S$ be a subset of $\R_d[\ux]$ and  $V:=\cZ(S)$ the corresponding real algebraic set. We denote by $\cI$  the ideal of $\R_d[\ux]$ generated by $S$ and by $\hat{\cI}$  the  ideal of  $f\in \R_d[\ux]$ which vanish on $V$. Clearly,   $\cZ(S)=\cZ(\cI)$ and 
$\cI\subseteq \hat{\cI}$. In general,  $\cI\neq \hat{\cI}$. (For instance, if $d=2$ and $S=\{x_1^2+x_2^2\}$, then $V=\{0\}$ and  $x_1^2\in \hat{\cI}$, but $x_1^2\notin\cI$.) 

It can be shown [BDRo, Theorem 4.1.4] that\,  $\cI=\hat{\cI}$\, if and only if\, $\sum p_j^2\in \cI$\, for finitely many $p_j\in \R_d[\ux]$ implies that $p_j\in \cI$ for all $j$. An ideal that obeys this property is called {\it real}. In particular,  $\hat{\cI}$ is  real. The ideal $\cI$ generated  by a single irreducible polynomial $h\in \R_d[\ux]$ is real if and only if $h$ changes its sign on $\R^d$, that is, there are $x_0,x_1\in \R^d$ such that $h(x_0)h(x_1)<0$, see [BCRo, Theorem 4.5.1].

The quotient algebra \index[sym]{RDVA@$\R[V]$}
\begin{align}\label{algregvar}
    \R[V]:=\R_d[\ux]/\hat{\cI}
\end{align} is called the algebra of {\it regular functions}\index{Real algebraic set! algebra of regular functions} on  $V$. Since  $\hat{\cI}$ is real, it follows that
\begin{align}\label{qminusq0}
    \sum \R[V]^2\cap \big(-\sum \R[V]^2\big) =\{0\}.  
\end{align}

\begin{example}\label{quadraticideal}
    Let us assume that the set ${\sf f}$ is of the form \begin{align*}{\sf f}=\{g_1,\cdots,g_l,h_1,-h_1,\dotsc,h_m,-h_m\}.\end{align*}
    If ${\sf g}:=\{g_1,\dotsc,g_l\}$ and  $\cI$ denotes the ideal of $\R_d[\ux]$ generated by $h_1,\dotsc,h_m$, then
    \begin{align}\label{quqadraticzero}
        \cK({\sf f})=\cK({\sf g})\cap \cZ(\cI),~~Q({\sf f})=Q({\sf g})+\cI, ~~{\rm and}~~ T({\sf f})=T({\sf g})+\cI.
    \end{align}

    We prove (\ref{quqadraticzero}).
    The first equality of (\ref{quqadraticzero}) and the inclusions  $Q({\sf f})\subseteq Q({\sf g})+\cI$ and $ T({\sf f})\subseteq T({\sf g})+\cI$ are clear  from the corresponding definitions. The  identity
    \begin{align*}
        ph_j=\frac{1}{4}
        [(p+1)^2h_j+(p-1)^2(-h_j)]\in Q({\sf f}),~~p\in \R_d[\ux],
    \end{align*}
    implies that $\cI\subseteq Q({\sf f})\subseteq T({\sf f})$. Hence $Q({\sf g})+\cI\subseteq Q({\sf f})$ and $T({\sf g})+\cI\subseteq T({\sf f}).$  $\hfill \circ$
\end{example}

Another important concept is introduced in the following definition.
\begin{definition}
    Let $C$ be a unital cone in  ${\sA}$. Define
    \begin{align*}
        {\sA}_b(C):=\{ a\in {\sA}: {\rm there ~exists~ a}~~\lambda >0 ~~{\rm such ~that}~~ \lambda - a \in C~~{\rm and}~~\lambda +a\in C\}.
    \end{align*}
    We shall say that  $C$ is  {\em Archimedean}\index{Quadratic module! Archimedean}\index{Semiring! Archimedean} if\, ${\sA}_b(C)={\sA}$, or equivalently, for every $a\in {\sA}$ there exists a $\lambda>0$ such that $\lambda -a\in C.$
\end{definition}
\begin{lemma}\label{boundedele1} Let $Q$ be a quadratic module of\, ${\sA}$\, and let $a\in {\sA}$. Then
    $a\in {\sA}_b(Q)$  if and only if\,   $\lambda^2 -a^2\in Q$ for some $\lambda >0$.
\end{lemma}
\begin{proof}
    If $\lambda \pm a\in Q$ for $\lambda>0$, then  
    \begin{align*}
        \lambda^2 -a^2= \frac{1}{2\lambda}\big[(\lambda+a)^2(\lambda -a) +(\lambda -a)^2(\lambda + a)\big] \in Q.\end{align*}
    Conversely, if $\lambda^2-a^2 \in Q$ and $\lambda >0$, then
    \begin{align*}\hspace{2,7cm}
        \lambda\pm a=\frac{1}{2\lambda}\big[(\lambda^2-a^2) +(\lambda\pm a)^2 \big] \in Q.
        \hspace{3cm} \Box
    \end{align*}
\end{proof}
\begin{lemma}\label{boundedele2} Suppose that $Q$ is a quadratic module or a semiring of ${\sA}$. 
    \begin{itemize}
        \item[\em (i)]\,  ${\sA}_b(Q)$ is a unital  subalgebra of ${\sA}$.
        \item[\em (ii)]\,\,  If the algebra ${\sA}$ is generated by elements $a_1,\dotsc,a_n$, then $Q$ is Archimedean if and only if each $a_i$ there exists a  $\lambda_i>0$ such that $\lambda_i\pm a_i \in Q$.
    \end{itemize}
\end{lemma}
\begin{proof} 
    (i): Clearly,  sums  and scalar multiples of elements of ${\sA}_b(Q)$ are again in ${\sA}_b(Q)$. It suffices to verify that this holds for the product of elements $a,b \in {\sA}_b(Q)$.
    
    First we suppose that $Q$ is a quadratic module.  By Lemma \ref{boundedele1},   there are $\lambda_1 >0$ and $\lambda_2>0$
    such that $\lambda_1^2-a^2$ and $\lambda_2^2-b^2$ are in $Q$. Then
    \begin{align*}
        (\lambda_1\lambda_2)^2 - (ab)^2= \lambda_2^2(\lambda_1^2-a^2)+ a^2(\lambda_2^2-b^2) \in Q,\end{align*}
    so that $ab\in {\sA}_b(Q)$ again by  Lemma \ref{boundedele1}.
    
    Now let $Q$ be a semiring. If $\lambda_1-a\in Q $ and $\lambda_2-b\in Q$, then 
    \begin{align*}
        \lambda_1\lambda_2\mp ab =\frac{1}{2}\big(( \lambda_1\pm a)(\lambda_2-b)+ (\lambda_2\mp a)(\lambda_2+b)\big) \in Q.
    \end{align*}
    
    (ii) follows at once from (i).
      \end{proof}  

By Lemma \ref{boundedele2}(ii), 
it suffices to check  the Archimedean condition $\lambda \pm a\in Q$   for  algebra generators. Often this simplifies  proving that $Q$ is Archimedean.
\begin{corollary}\label{archirux}
    For a quadratic module $Q$  of\, $\R_d[\ux]$ the following are equivalent:
    \begin{itemize}
        \item[\em (i)]~ $Q$ is Archimedean.
        \item[\em (ii)]~\,  There exists a number $\lambda >0$ such that $\lambda -\sum_{k=1}^d x_k^2\in Q$.
        \item[\em (iii)]~\,  For any $k=1,\dotsc,d$ there exists a  $\lambda_k>0$ such that $\lambda_k-x_k^2\in Q$.
    \end{itemize}
\end{corollary}
\begin{proof}
    (i)$\to$(ii) is clear by definition. If $\lambda -\sum_{j=1}^d x_j^2\in Q$, then \begin{align*}\lambda- x_k^2=\lambda -\sum\nolimits_j x_j^2~ +~ \sum\nolimits_{j\neq k}x_j^2 \in Q.\end{align*} This  proves (ii)$\to$(iii). Finally, if (iii) holds,  then $x_k\in {\sA}_b(Q)$ by Lemma \ref{boundedele1} and hence ${\sA}_b(Q)={\sA}$ by Lemma \ref{boundedele2}(ii). Thus, (iii)$\to$(i).
      \end{proof}  

Note that $S=\R_+ \cdot 1$ is a semiring, so semirings could be rather ``small''. 
\begin{definition} A semiring $S$ is called \emph{generating} if\, $A=S-S$.
\end{definition}
An Archimedean semiring is always generating, since $a=\lambda  -(\lambda  -a)$ for $a\in A$ and $\lambda \in \R$.

\begin{corollary}\label{archicompact}
    If the quadratic module $Q({\sf f})$ of $\R_d[\ux]$ is Archimedean, then the set $\cK({\sf f})$ is compact.
\end{corollary}
\begin{proof}
    By the respective definitions, polynomials of  $Q({\sf f})$ are nonnegative on $\cK({\sf f})$.
    Since $Q({\sf f})$ is Archimedean, $\lambda -\sum_{k=1}^d x_k^2\in Q({\sf f})$ for some $\lambda >0$ by Corollary \ref{archirux}, so  $\cK({\sf f})$ is contained in the ball centered at the origin with radius $\sqrt{\lambda}$. 
      \end{proof}  

The converse of Corollary \ref{archicompact}  does not hold, as the following example shows.
(However, it does hold for the preordering $T({\sf f})$ as shown by Proposition \ref{prearchcom} below.)
\begin{example}\label{archnoncompact}
    Let $f_1=2x_1-1$, $f_2=2x_2-1$, $f_3=1-x_1x_2$. Then the set $\cK({\sf f})$ is compact, but $Q({\sf f})$ is not Archimedean (see [PD, p. 146] for a proof). $\hfill \circ$
\end{example}

The following  separation result will be used in Sections \ref{reparchimodiules} and \ref{Operator-theoreticappraochtothemomentprblem}. 

\begin{proposition}\label{eidelheit}
    Let $C$ be an Archimedean  unital cone of ${\sA}$. If $a_0\in {\sA}$ and $a_0\notin C$, there exists a $C$-positive linear functional $\varphi$ on ${\sA}$ such that $\varphi(1)=1$ and $\varphi(a_0)\leq 0$. The functional $\varphi$ may be chosen as an extremal functional of the dual cone \begin{align}\label{cwedge}
        C^\wedge:=\{ L\in A^*: L(c)\geq 0~~~ \textup{for}~~ c\in C\, \}.
    \end{align}
\end{proposition}
\begin{proof}
    Let $a\in {\sA}$ and choose $\lambda >0$ such that\, $\lambda \pm a\in C$. If\,  $0<\delta \leq\lambda^{-1}$,\, then $\delta^{-1} \pm a\in C$ and hence $1\pm \delta a \in C$. Thus   $1$ is an internal point of $C$ and an order unit for  $C$. 
    Therefore a separation theorem for convex sets (see e.g. Proposition C.5 in [Sm20]) applies, so  there exists an extremal functional $\varphi$ of $C^\wedge$ such that $\varphi(1)=1$ and $\varphi(a_0)\leq 0$. (Without the extremality of $\varphi$ this result follows also  from Eidelheit's separation Theorem A.27.)
      \end{proof}  

\begin{example}
    Let ${\sA}=\R_d[\ux]$ and let $K$ be a closed subset of $\R^d$. If $C$ is the preordering ${\Pos} (K)$ of nonnegative polynomials on $K$, then ${\sA}_b(C)$ is just the set of bounded polynomials on $K$. Hence $C$ is Archimedean if and only if $K$ is compact. $\hfill \circ$
\end{example}

Recall from Definition 1.13 that  $\hat{{\sA}}$ denotes the set of  characters of the real algebra ${\sA}$, that is,  the set of unital algebra homomorphism $\chi:\sA\to \R$. 
%Note that $\chi(1)=1$ for all $\chi\in \hat{\sA}$. 

For  a subset   $C$  of ${\sA}$ we define\index[sym]{KCQA@$\cK(C)$}  
\begin{align}\label{definionkq}
    \cK(C):=\{ \chi\in \hat{{\sA}}:  \chi(c)\geq 0 ~~ {\rm for~all}~c\in C\}.
\end{align}
\begin{example}\label{ckcpoly} ${\sA}=\R_d[\ux]$\\ Then $\hat{A}$ is the set of   evaluations $\chi_t(p)=p(t), p\in {\sA}$, at points of $\R^d$. As usual,
    we identify $\chi_t$ and $t$, so that $\hat{A}\cong \R^d$. Then, if $C$ is the quadratic module $Q({\sf f})$
    defined by (\ref{quadraticqf}) or\, $C$ is the semiring $S({\sf f})$ defined by (\ref{semiringf}) or\, $C$ is the preordering $T({\sf f})$ defined by (\ref{preorderingtf}), 
    the set $\cK(C)$ is just the semi-algebraic  set $\cK({\sf f})$
    given by (\ref{basicclosed}). \hfill $\circ$
\end{example}
Let $C$ be a quadratic module or a semiring. The set $C^{\sat}={\Pos}(\cK(C))$ of all $f\in {\sA}$ which  are nonnegative on the set $\cK(C)$ is obviously a preordering of ${\sA}$ that contains $C$.
Then $C$ is called {\it saturated}\index{Quadratic module! saturated ! Semiring!} if\, $C=C^{\sat}$, that is, if $C$ is equal to its {\it saturation}\, $Q^{\sat}$. 

Real algebraic geometry is treated in the books [BCRo], [PD], [Ms1]; a  recent survey on positivity and sums of squares is given in [Sr3].

%\end{document}

\section{Localizing functionals and supports of representing measures}\label{localizing}

Haviland's Theorem 1.12 shows that there is a close  link between positive polynomials and the moment problem. However, in order to apply this result reasonable descriptions of positive, or at least of strictly positive, polynomials are needed.

Recall that the moment problem for a functional $L$ on the interval $[a,b]$  is solvable if and only if $L(p^2+(x-a)(b-x)q^2)\geq 0$ for all $p,q \in \R[x]$. This condition means that two infinite  Hankel matrices are positive semidefinite and this holds if and only if all  principal minors of these matrices are nonnegative. In the multidimensional case we are trying to find similar  solvability criteria. For this it is natural to  consider sets that are  defined by finitely many polynomial inequalities $f_1(x)\geq 0,\dotsc, f_k(x)\geq 0$. These are precisely the basic closed semi-algebraic sets $\cK({\sf f})$, so we have entered  the setup of real algebraic geometry. 

Let us fix a semi-algebraic set   $\cK({\sf f})$. Let $L$ be a $\cK({\sf f})$-moment functional, that is, $L$ is of the form\, $L(p)=L^\mu(p)\equiv \int p\,d\mu$ for  $p \in \R_d[\ux]$,\, where $\mu$ is a Radon measure   supported on  $\cK({\sf f})$. If $g\in \R_d[x]$ is  nonnegative  on   $\cK({\sf f})$, then obviously 
\begin{align}\label{hcondition}
    L(gp^2)\geq 0 \quad{\rm for~ all}\quad p \in \R_d[\ux],
\end{align}
so (\ref{hcondition}) is  a\, {\it necessary}\, condition for $L$ being a  $\cK({\sf f})$-moment functional.

The overall strategy in this chapter and the next is to solve the  $\cK({\sf f})$-moment problem by  {\it   finitely many sufficient} conditions of the form (\ref{hcondition}). That is, our aim is to ``find"  nonnegative polynomials $g_1,\dotsc,g_m$\, on $\cK({\sf f})$ such that the following  holds: 

{\it Each linear functional $L$ on $\R_d[\ux]$ which  satisfies   condition (\ref{hcondition}) for $g=g_1,\dotsc,g_m$ and $g=1$ is a 
    $\cK({\sf f})$-moment functional.} (The polynomial $g=1$ is  needed in order to ensure that $L$ itself is a positive  functional.)

In general it is not sufficient  to take only the polynomials $f_j$ themselves as $g_j$. For  our main results (Theorems \ref{mpschm} and 13.10), the positivity  of the functional on the  preordering $T({\sf f})$ is assumed. This means that condition (\ref{hcondition}) is required for {\it all} mixed products  $g=f_1^{e_1}\cdots f_k^{e_k}$, where $e_j\in \{0,1\}$ for $j=1,\dotsc,k$.

\begin{definition}\label{localizedfunctional}
    Let $L$ be a  linear functional on  $\R_d[\ux]$ and let $g\in\R_d[\ux]$. The  linear functional $L_g$\index[sym]{LAgA@$L_g$}  on  $\R_d[\ux]$ defined by $L_g(p)=L(gp),\,  p\in\R_d[\ux]$, is called the \emph{localization} of $L$ at $g$  or simply the \emph{localized functional}.\index{Functional! localized, localization}
\end{definition} 
Condition (\ref{hcondition}) means  the localized functional $L_g$ is  a positive linear functional on  $\R_d[\ux].$ Further, if $L$ comes from a measure $\mu$ supported on  $\cK({\sf f})$ and $g$ is nonnegative  on $\cK({\sf f})$, then 
\begin{align*}
    L_g(p)=L(gp)=\int_{\cK({\sf f})} p(x)\, g(x)d\mu(x), ~~p\in\R_d[\ux],
\end{align*}
that is, 
$L_g$ is given by the measure $\nu$ on $\cK({\sf f})$  defined by $d\nu=g(x)d\mu$. 

Localized functionals will play an important role throughout our  treatment. They are used to localize the support of the measure (see Propositions  \ref{cqfimpliessuppckf} and \ref{zerosetsupport} and Theorem 14.25) or to derive determinacy criteria (see Theorem 14.12).

Now we  introduce two other objects  associated with the functional $L$ and the polynomial $g$.  Let $s=(s_\alpha)_{\alpha \in \N_0^d}$ be the $d$-sequence given by  $s_\alpha=L(x^\alpha)$ and  write $g=\sum_\gamma g_\gamma x^\gamma$. Then we define a $d$-sequence $g(E)s=((g(E)s)_\alpha)_{\alpha\in \N_0^d}$ by 
\begin{align*}(g(E)s)_\alpha:=\sum\nolimits_\gamma ~ g_\gamma s_{\alpha+\gamma},~~\alpha\in \N_0^d,\end{align*}  
and an infinite matrix\, $H(gs)=(H(gs)_{\alpha,\beta})_{\alpha,\beta\in \N_0^d} $\, over\, $\N_0^d\times \N_0^d$ with entries\index[sym]{HAgAsA@$H(gs)$} 
\begin{align}\label{localhaneklentries}
    H(gs)_{\alpha,\beta}:= \sum\nolimits_\gamma ~ g_\gamma  s_{\alpha+\beta+\gamma},~~\alpha,\beta\in \N_0^d.
\end{align}
Using these definitions
for  $p(x)=\sum_\alpha
a_\alpha x^{\alpha}\in \R_d[\ux]$ we compute
\begin{align}\label{hcondition1}
    L_s(gp^2)=\sum_{\alpha,\beta,\gamma} a_\alpha a_\beta g_\gamma s_{\alpha+\beta+\gamma} =\sum_{\alpha,\beta} a_\alpha a_\beta (g(E)s)_{\alpha+\beta}=\sum_{\alpha,\beta}\, a_\alpha a_\beta H(gs)_{\alpha,\beta} .
\end{align}
This shows that $g(E)s$ is the\, $d$-sequence for the functional $L_g$ and  $H(gs)$ is a Hankel matrix for the sequence $g(E)s$. The matrix $H(gs)$ is called the {\it localized Hankel matrix}\index{Hankel matrix! localized} of $s$ at $g$.
\begin{proposition}\label{equvialentsolvmp}
    Let\, $Q({\sf g})$\, be  the quadratic module generated by the finite subset\, ${\sf g}=\{g_1,\dotsc,g_m\}$  of\,  $\R_d[\ux]$. Let $L$ be a linear functional on $\R_d[\ux]$ and $s=(s_\alpha)_{\alpha \in \N_0^d}$  the $d$-sequence defined by\, $s_\alpha=L(x^\alpha).$  
    Then the following are equivalent:
    \begin{itemize}
        \item[\em (i)]~ $L$ is a $Q({\sf g})$-positive linear functional on\,  $\R_d[\ux]$.
        \item[\em (ii)]~  $L, L_{g_1},\dotsc,L_{g_m}$ are positive linear  functionals on\, $\R_d[\ux]$.
        \item[\em (iii)]~\,  $s, g_1(E)s,\dotsc,g_m(E)s$ are positive semidefinite $d$-sequences.
        \item[\em (iv)]~\, $H(s),H(g_1s),\dotsc, H(g_ms)$ are positive semidefinite matrices.
    \end{itemize}
\end{proposition} 
\begin{proof}
    The equivalence of (i) and (ii) is immediate from the
    definition (\ref{quadraticqf}) of the quadratic module $Q({\sf g})$ and  Definition \ref{localizedfunctional} of the  localized functionals $L_{g_j}$. 
    
    By Proposition  2.7, a linear functional is positive if and only if the corresponding sequence is positive semidefinite, or equivalently,  the  Hankel matrix is positive semidefinite. 
    By   (\ref{hcondition1}) this gives the equivalence of (ii),  (iii),  and (iv). 
      \end{proof}  

The solvability conditions in the  existence theorems for the moment problem 
in this chapter and the next 
are given in the form (i) for some  finitely  generated quadratic module or   preordering. This  means  that condition (\ref{hcondition}) is  satisfied for  finitely many polynomials $g$.
Proposition \ref{equvialentsolvmp} says there  are various {\it equivalent} formulations of these
solvability criteria: They  can be  expressed in the language  of real algebraic geometry (in terms of quadratic modules, semirings or preorderings),
of $*$-algebras (as positive  functionals on  $\R_d[\ux]$), of  matrices (by the positive semidefiniteness  of Hankel matrices) or of sequences (by the positive semidefiniteness of sequences).

The next proposition contains a useful  criterion for localizing  supports  of  representing measures. We denote by  $\cM_+(\R^d)$\index[sym]{MCplusRD@$\cM_+(\R^d)$} the set of Radon measure $\mu$ on $\R^d$ for which all moments are finite, or equivalently,  $\int |p(x)|\, d\mu < \infty$\, for all $p\in \R_d[\ux]$.
\begin{proposition}\label{locquader}
    Let $\mu\in \cM_+(\R^d)$ and let $s$ be the moment sequence of $\mu$. Further, let $g_j\in \R_d[\ux]$ and  $c_j\geq 0$  be given for $j=1,\dotsc,k.$ Set 
    \begin{align}\label{defsetK}
        \cK=\{ x\in \R^d:  |g_j(x)|\leq c_j ~~~  \text{for} ~~j=1,\dotsc,k\}.
    \end{align}
    Then we have\, ${\rm supp} ~\mu \subseteq \cK$ if and only if   there exist constants $M_j>0$  such that 
    \begin{align}\label{condtionL_my}
        L_s(g_j^{2n})\leq M_jc_j^{2n}~~~\text{for}~~ n\in \N,~j=1,\dotsc,k.
    \end{align}
\end{proposition}
\begin{proof}
    The  only if part is obvious. We prove the  if direction and
    slightly modify the argument used in the proof of Proposition 4.1.

    Let  $t_0 \in \R^d\backslash \cK$. Then there is an index\, $j=1,\dotsc,k$ such that $|g_j(t_0)|>c_j$. Hence there exist a number $\lambda>c_j$ and a ball $U$ around $t_0$ such that $|g_j(t)|\geq \lambda$ for $t\in U$. For $n\in \N$ we then derive
    \begin{align*}
        \lambda^{2n}\mu(U)\leq \int_U g_j(t)^{2n}\, d\mu(t)\leq \int_{\R^d} g_j(t)^{2n}\, d\mu(t)=L_s(g_j^{2n})\leq M_jc_j^{2n}.
    \end{align*}
    Since $\lambda>c_j $, this is only possible for all $n\in \N$ if\, $\mu(U)=0$. Therefore, $t_0\notin{\rm supp} ~\mu$. This  proves that ${\rm supp} ~\mu \subseteq \cK.$
      \end{proof}  

We state the special case  $g_j(x)=x_j$ of Proposition \ref{locquader} separately as
\begin{corollary}\label{suppndiminterval}
    Suppose  $c_1>0,\dotsc,c_d>0$. 
    A measure  $\mu\in \cM_+(\R^d)$ with moment sequence $s$ is supported on the $d$-dimensional interval~ $[-c_1,c_1]\times\dots \times [-c_d,c_d]$~ if and only if there are positive constants $M_j$ such that
    \begin{align*}
        L_s(x_j^{2n})\equiv s_{(0,\dotsc,0,1,0,\dotsc,0)}^{2n}\leq M_jc_j^{2n}~~~~\text{for}~~n\in\N,~j=1,\dotsc,d.\end{align*}
\end{corollary}
The following two  propositions are   basic results about the moment problem on {\it compact} sets. Both follow  from   Weierstrass' theorem on  approximation of continuous functions by polynomials. 
\begin{proposition}\label{detcomacpcase}
    If $\mu\in \cM_+(\R^d)$ is supported on a compact set,  then $\mu$ is determinate. In particular, if $K$ is a compact subset of $\R^d$, then each $K$-moment sequence,  so each measure $\mu\in \cM(\R^d)$ supported on $K$, is determinate.
\end{proposition} 
\begin{proof}
    Let $\nu\in \cM_+(\R^d)$ be a measure having the same moments and so the same moment functional $L$ as $\mu$.  Fix $h\in C_c(\R^d,\R)$. We choose a compact $d$-dimensional interval $K$ containing the supports of  $\mu$ and  $h$. From Corollary \ref{suppndiminterval} it follows that ${\supp}\, \nu\subseteq K$. By  Weierstrass' theorem, there is a sequence $(p_n)_{n\in \N}$ of polynomials $p_n\in \R_d[\ux]$ converging to $h$ uniformly on $K$. Passing to the limits in the equality 
    \begin{align*}
        \int_K p_n\, d\mu=L(p_n)=\int_K p_n\, d\nu
    \end{align*}
    we get $\int  h\,d\mu =\int h \,d\nu$. Since this holds for  all $h\in C_c(\R^d,\R)$, we have $\mu=\nu$.
    \end{proof}  

\begin{proposition}\label{cqfimpliessuppckf}
    Suppose that $\mu\in \cM_+(\R^d)$ is supported on a compact set. Let   ${\sf f}=\{f_1,\dotsc,f_k\}$ be a finite subset of $\R_d[\ux]$ and assume that the moment functional  defined by $L^\mu(p) =\int p\, d\mu$, $p\in \R_d[\ux]$, is $Q({\sf f})$-positive. Then ${\rm supp} ~\mu \subseteq \cK({\sf f}).$
\end{proposition}
\begin{proof}
    Suppose that $t_0 \in \R^d\backslash \cK({\sf f})$.
    Then there exist a number $j\in \{1,\dotsc,k\}$,  a ball $U$ with radius $\rho >0$ around $t_0$, and a number $ \delta>0$ such that $f_j \leq-\delta$ on $2 U$. 
    We define a continuous function $h$ on $\R^d$ by $h(t)=\sqrt{2\rho{-}|| t -t_0||}$\, for\, $|| t - t_0|| \leq 2\rho$\, and $h(t)=0$ otherwise and take a compact $d$-dimensional interval $K$  containing $2 U$ and ${\rm supp} ~\mu$. By  Weierstrass' theorem, there is a sequence of polynomials $p_n \in \R_d[\ux]$ converging to $h$ uniformly on  $K$. 
    Then\, $f_jp_n^2\to f_jh^2$\,   uniformly on $K$ and hence
    \begin{align}\label{estiamtefkU}
        \lim_{n}\,    L^\mu(& f_j p_n^2) =\int_K (\lim_n\, f_j p_n^2)\, d\mu= \int_K~ f_j h^2\,  d\mu\nonumber = \int_{2 U} f_j(t)(2 \rho{-}|| t-t_0 ||)\, d\mu(t)\\ &
        \leq  \int_{2U} -\delta(2 \rho{-}|| t-t_0 ||)\,  d\mu\leq -\int_U\delta  \rho\, d\mu(t)=
        -\delta \rho \mu(U) .
    \end{align} 
    Since $L^\mu$ is $Q({\sf f})$-positive, we have\, $L^\mu(f_j p_n^2)\geq 0$. Therefore, $\mu(U)=0$ by (\ref{estiamtefkU}), so that   $t_0\notin{\rm supp} ~\mu$. This proves  that\, ${\rm supp} ~\mu \subseteq \cK({\sf f}).$
    \end{proof}

The assertions of Propositions \ref{detcomacpcase} and \ref{cqfimpliessuppckf} are no longer valid if the compactness assumptions are omitted. But  the counterpart of Proposition \ref{cqfimpliessuppckf} for zero sets of ideals holds without any compactness assumption.
\begin{proposition}\label{zerosetsupport}
    Let $ \mu\in \cM_+(\R^d)$ and let $\cI$ be an ideal of $\R_d[\ux]$.  If the moment functional $L^\mu$ of\, $\mu$ is $\cI$-positive,
    then $L^\mu$ annihilates $\cI$ and \,  ${\rm supp}\, \mu \subseteq \cZ(\cI )$.\\
    (As usual,  $\cZ(\cI)=\{x\in \R^d:p(x)=0~~ \text{for}~  p\in \cI\}$ is the zero set of $\cI$.)
\end{proposition} 
\begin{proof}
    If $p\in \cI$, then $-p\in \cI$ and hence $L^\mu(\pm p)\geq 0$ by the $\cI$-positivity of $L^\mu$, so that $L^\mu(p)=0$. That is, $L^\mu$ annihilates $\cI$.
    
    Let $p\in \cI$. Since $p^2\in \cI$, we have $L^\mu(p^2)=\int p^2\, d\mu=0$. Therefore, from  Proposition \ref{zerosuppproposition} it follows that 
    ${\rm supp}\, \mu\subseteq \cZ(p^2)=\cZ(p)$. Thus,  ${\rm supp}\, \mu \subseteq \cZ(\cI )$.
    \end{proof}  

For a linear functional $L$ on $\R_d[\ux]$ we define\index[sym]{NCplusL@$\cN_+(L)$}
\begin{align*}
    \cN_+(L) :=\{ f\in {\Pos}(\R^d): L(p)=0\, \}.
    %\cV_+(L)=\cZ(\cN_+(L))
\end{align*}
\begin{proposition}
    Let $L$ be a moment functional on $\R_d[\ux]$, that is, $L=L^\mu$ for some   $\mu\in \cM_+(\R^d)$.  Then the ideal $\cI_+(L)$  of $\R_d[\ux]$ generated by $\cN_+(L)$ is annihilated by $L$ and the support of each representing  measure of $L$ is contained in $\cZ(\cI_+(L))$.
\end{proposition}
\begin{proof}
    Let $\nu$ be an arbitrary representing measure of $L$. If $f\in \cN_+(L)$, then we have\, $L(f)=\int f(x)\, d\nu=0$. Since $f\in {\Pos}(\R^d)$, Proposition \ref{zerosuppproposition} applies and yields  ${\rm supp}\, \nu\subseteq \cZ(f)$. Hence ${\rm supp}\, \nu\subseteq \cZ(\cN_+(L)))=\cZ(\cI_+(L)).$
    In particular, the inclusion ${\rm supp}\, \nu\subseteq \cZ(\cI_+(L))$ implies that $L=L^\nu$ annihilates $\cI_+(L).$ 
    \end{proof}  

\section{The moment problem on compact semi-algebraic sets and the strict Positivstellensatz}\label{momentproblemstrictpos}
The solutions of one-dimensional moment problems  have been derived from
descriptions of nonnegative polynomials as weighted sums of squares. The counterparts of the latter in  the multidimensional case are  the so-called ``Positivstellens\"atze"  of real algebraic geometry. In general these results require denominators (see Theorem  \ref{krivinestengle}), so they do not yield reasonable criteria for solving  moment problems.  However, for {\it strictly positive} polynomials on {\it compact}  semi-algebraic sets $\cK({\sf f})$ there are {\it denominator free} Positivstellens\"atze (Theorems \ref{schmps} and \ref{archmedps}) which provides  solutions of  moment problems. Even more, it turns out that there is a close interplay between this type of  Positivstellens\"atze  and  moment problems on compact semi-algebraic sets, that is,  existence results for the moment problem can be derived from  Positivstellens\"atze  and vice versa.

We state the main technical steps of the proofs separately as  Propositions \ref{technkrivinestengle}--\ref{continuityLprop}.
Proposition \ref{continuityLprop} is also used  in a crucial manner in the proof of   Theorem 13.10 below.

Suppose that ${\sf f}=\{f_1,\dotsc,f_k\}$ is a finite subset of $\R_d[\ux]$. Let $B(\cK({\sf f}))$\index[sym]{BAKCfB@$B(\cK({\sf f}))$} denote  the algebra of all polynomials of $\R_d[\ux]$ which are bounded on the set $\cK({\sf f})$.
\begin{proposition}\label{technkrivinestengle} 
    Let $g\in B(\cK({\sf f}))$ and $\lambda >0$. If\, $\lambda^2>g(x)^2$\, for all $x\in  \cK({\sf f})$, then there exists a $p \in T({\sf f})$
    such that 
    \begin{align}\label{p2n2}
        g^{2n}  \preceq \lambda^{2n+2}p ~~~\text{for}~~~n\in \N.
    \end{align}
\end{proposition}
\begin{proof} 
    By the Krivine--Stengle Positivstellensatz (Theorem \ref{krivinestengle}(i)), applied to the positive polynomial $\lambda^2-g^2$ on  $\cK({\sf f})$, there exist polynomials
    $p,q \in T({\sf f})$ such that
    \begin{align}\label{stengle}
        p(\lambda^2-g^2) =1 +q.
    \end{align}
    Since $q \in T({\sf f})$ and $T({\sf f})$ is a quadratic module, $g^{2n}(1+q) \in T({\sf f})$ for $n \in \N_0$. Therefore, using (\ref{stengle}) we conclude  that
    \begin{align*}g^{2n+2}p = g^{2n} \lambda^2 p-g^{2n}(1+q) \preceq g^{2n} \lambda^2 p.\end{align*} By induction it follows that 
    \begin{align}\label{p2n1}
        g^{2n} p \preceq \lambda^{2n} p.
    \end{align}
    Since $g^{2n}(q+pg^2) \in T({\sf f})$, using first (\ref{stengle}) and then  (\ref{p2n1}) we derive
    \begin{align*}
        \hspace{0,7cm}g^{2n} \preceq g^{2n} + g^{2n}(q+pg^2) = g^{2n}(1+q+pg^2)=g^{2n}\lambda^2 p \preceq \lambda^{2n+2}p\,. \hspace{1,1cm} \Box
    \end{align*}
\end{proof}

\begin{proposition}\label{prearchcom}
    If the set $\cK({\sf f})$ is compact, then the 
    associated 
    preordering $T({\sf f})$ is Archimedean.
\end{proposition}
\begin{proof} Put $g(x):=(1+x_1^2)\cdots(1+x_d^2)$.
    Since $g$ is bounded on the compact set $\cK({\sf f})$, we have $\lambda^2 >g(x)^2$ on  $\cK({\sf f})$ for some $\lambda>0$. Therefore, by Proposition \ref{technkrivinestengle} there exists a $p \in T({\sf f})$ such that (\ref{p2n2}) holds.  
    
    Further, for any multiindex $\alpha\in \N_0^d$, $|\alpha| \leq k$, $k \in \N$, we obtain
    \begin{align}\label{pk}
        \pm 2x^\alpha \preceq x^{2\alpha} + 1 \preceq \sum_{|\beta| \leq k} x^{2\beta} = g^k.
    \end{align}
    Hence there exist numbers $c >0$ and $k \in \N$ such that $p
    \preceq 2c g^k$. Combining the latter with $g^{2n}  \preceq \lambda^{2n+2}p$ by~(\ref{p2n2}), we get
    $g^{2k} \preceq \lambda^{2k+2} 2c g^k$ and so
    \begin{align*}(g^k{-}\lambda^{2k+2}c)^2 \preceq (\lambda^{2k+2}c)^2{\cdot}1.\end{align*} Hence, by
    Lemma \ref{boundedele1}, $g^k{-}\lambda^{2k+2}c \in {\sA}_b(T({\sf f}))$ and so  $g^k \in {\sA}_b(T({\sf f}))$, where ${\sA}:=\R_d[\ux]$. Since $\pm x_j \preceq g^k$ by
    (\ref{pk}) and $g^k \in {\sA}_b(T({\sf f}))$, we obtain $x_j \in {\sA}_b(T({\sf f}))$ for $j=1,{\cdots},d$. Now from Lemma \ref{boundedele2}(ii) it follows that ${\sA}_b(T({\sf f}))=
    {\sA}$. This  means that $T({\sf f})$ is Archimedean.
    \end{proof}  

\begin{proposition}\label{continuityLprop}
    Suppose that $L$ is a\, $T({\sf f})$-positive linear functional on $\R_d[\ux]$. 
    \begin{itemize}
        \item[\em (i)]\,\,\,  If\, $g\in B(\cK({\sf f}))$ and $\|g \|_\infty$ denotes the  supremum of  $g$ on $\cK({\sf f}),$  then
        \begin{align}\label{conitnuityL}
            |L(g)|\leq L(1) ~\|g\|_\infty .
        \end{align}
        \item[\em (ii)]~ If\, $g\in B(\cK({\sf f}))$ and $g(x)\geq 0$ for $x\in \cK({\sf f})$, then $L(g)\geq	 0$.
    \end{itemize}
\end{proposition}
\begin{proof}
    (i):  Fix $\varepsilon >0$ and put $\lambda:= \parallel g \parallel_\infty +\varepsilon$. 
    We define a real sequence $s=(s_n)_{n\in \N_0}$ by  $s_n:=L(g^n)$. Then $L_s(q(y))=L(q(g))$ for  $q\in \R[y]$. For any $p\in \R[y]$, we have $p(g)^2\in \sum \R_d[\ux]^2\subseteq T({\sf f})$ and hence $L_s(p(y)^2)=L(p(g)^2)\geq 0$, since $L$ is $T({\sf f})$-positive.
    Thus, by  Hamburger's theorem 3.8, there exists a Radon measure $\nu$ on $\R$ such that $s_n= \int_\R t^n d\nu(t)$, $n \in \N_0$. 
    
    For $\gamma >\lambda$ let $\chi_\gamma$ denote the characteristic function of the set $(-\infty,-\gamma]\cup[\gamma,+\infty)$. Since $\lambda^2-g(x)^2>0$ on $\cK({\sf f})$,  we have $g^{2n} \preceq \lambda^{2n+2}p$\, by equation (\ref{p2n2}) in Proposition \ref{technkrivinestengle}. Using the $T({\sf f})$-positivity of $L$  we derive
    \begin{align}\label{gamma2nlp}
        \gamma^{2n} \int_\R \chi_\gamma(t) ~d\nu(t) \leq \int_\R t^{2n} d\nu(t) =s_{2n}= L(g^{2n}) \leq \lambda^{2n+2} L(p)
    \end{align}
    for all $n\in \N$. Since $\gamma >\lambda$, (\ref{gamma2nlp}) implies that $\int_\R \chi_\gamma(t)~d\nu(t) =0$. Therefore, ${\rm supp}~\nu \subseteq [-\lambda,\lambda]$. (The preceding argument has been already used in the proof of Proposition \ref{locquader} to obtain a similar conclusion.) Therefore, applying the Cauchy--Schwarz inequality for $L$ we derive
    \begin{align*}
        |L(g)|^2 &\leq L(1)L(g^2) 
        = L(1) s_2=
        L(1)\int_{-\lambda}^\lambda~ t^2 ~d\nu(t)\\& \leq L(1)\nu(\R)\lambda^2 = L(1)^2\lambda^2=L(1)^2(\parallel g \parallel_\infty + \varepsilon)^2.
    \end{align*}
    Letting $\varepsilon \to +0$, we get\, $|L(g)| \leq L(1)\parallel g \parallel_\infty$.

    (ii): Since $g\geq 0$ on $\cK({\sf f})$, we clearly have\, $\|\,1\cdot \|g\|_\infty-2\,g\|_\infty = \|g\|_\infty.$ Using this equality and (\ref{conitnuityL}) we conclude that
    \begin{align*}
        L(1) \|g\|_\infty - 2\,L(g)=L( 1 \cdot\|g\|_\infty - 2\,g) \leq L(1)\| 1\cdot \|g\|_\infty -2\,g\|_\infty =L(1)\|g\|_\infty ,
    \end{align*}
    which in turn implies that\, $L(g)\geq 0$.
    \end{proof}  

The following theorem is the {\it  strict Positivstellensatz}\index{Positivstellensatz! strict} for compact basic closed semi-algebraic sets 
$\cK({\sf f}).$
\begin{theorem}\label{schmps}
    Let ${\sf f}=\{f_1,\dotsc,f_k\}$ be a finite subset of $\R_d[\ux]$ and let\, $h\in \R[x]$. If the set $\cK({\sf f})$ is compact and $ h(x) > 0$ for all $x \in \cK({\sf f})$, then $h \in T({\sf f})$.
\end{theorem}
\begin{proof} Assume to the contrary that $h$ is not in $T({\sf f})$. By Proposition \ref{prearchcom}, $T({\sf f})$ is Archimedean. Therefore, by Proposition \ref{eidelheit}, there exists a $T({\sf f})$-positive linear functional $L$ on ${\sA}$ such that $L(1)=1$ and $L(h) \leq 0$. 
    Since $h >0$ on the compact set $\cK({\sf f})$, there is a positive number $\delta$ such that $h(x)-\delta> 0$ for all $x\in\cK({\sf f})$. We extend the continuous function $\sqrt{h(x)-\delta}$\,  on $\cK({\sf f})$ to a continuous function on some compact $d$-dimensional interval containing $\cK({\sf f})$. Again by
    the classical Weierstrass theorem,\,  $\sqrt{h(x)-\delta}$\, is the uniform
    limit on $\cK({\sf f})$  of a sequence $(p_n)$ of polynomials $p_n \in \R_d[\ux]$. Then\, $p_n^2-h+\delta\to 0$ uniformly on $\cK({\sf f})$, that is, $\lim_{n} \parallel p_n^2 -h +\delta \parallel_\infty =0$. Recall that $ B(\cK({\sf f}))=\R_d[\ux]$, since $\cK({\sf f})$ is compact. Hence\, $\lim_{n} L(p_n^2-h+ \delta)=0$\, by the inequality (\ref{conitnuityL})  in Proposition \ref{continuityLprop}(i). But, since $L(p_n^2) \geq 0$, $L(h)\leq 0$, and $L(1)=1$, we have\, $L(p_n^2-h+\delta)\geq \delta >0$ which is the desired contradiction. This completes the proof of the theorem.
    \end{proof}  

The next result  
gives a solution of the $\cK({\sf f})$-moment problem for compact basic closed semi-algebraic sets.\index{Moment problem! for compact semi-algebraic sets}
\begin{theorem}\label{mpschm}
    Let  ${\sf f}=\{f_1,\dotsc,f_k\}$ be a finite subset of $\R_d[\ux]$. If  the set $\cK({\sf f})$ is compact, then each $T({\sf f})$-positive linear functional $L$ on $\R_d[\ux]$ is a $\cK({\sf f})$-moment functional.
\end{theorem}
\begin{proof}
    Since $\cK({\sf f})$ is compact, $ B(\cK({\sf f}))=\R_d[\ux]$. Therefore, it suffices to combine Proposition \ref{continuityLprop}(ii) with Haviland's Theorem 1.12.
    \end{proof}  

\begin{remark} Theorem \ref{mpschm} was obtained from Proposition \ref{continuityLprop}(ii) and Haviland's Theorem 1.12. Alternatively, it can derived from Proposition \ref{continuityLprop}(i) combined with  Riesz' representation theorem. Let us sketch this proof. By (\ref{conitnuityL}), the  functional $L$ on $\R_d[\ux]$ is $\|\cdot\|_\infty$- continuous.  Extending $L$ to $C(\cK({\sf f}))$ by the Hahn--Banach theorem and applying   Riesz' representation theorem for continuous linear functionals,  $L$ is given by a signed Radon measure on $\cK({\sf f})$. Setting $g=1$ in (\ref{conitnuityL}),  it follows that $L$, hence the extended functional, has  the norm $L(1)$. It is not difficult to show that this implies that the representing measure is positive. $\hfill \circ$
\end{remark}

The shortest path to Theorems \ref{schmps} and \ref{mpschm} is probably  to use Proposition \ref{continuityLprop} as we have done. However, in order to emphasize the interaction between both theorems  and so in fact  between the moment problem and real algebraic geometry we now derive each of these theorems from the other.
\smallskip

{\it Proof of Theorem \ref{mpschm} (assuming Theorem \ref{schmps}):}\\
Let $h\in \R_d[\ux]$. If $h(x)>0$ on $\cK({\sf f})$, then $h\in T({\sf f})$ by  Theorem \ref{schmps} and so $L(h)\geq 0$ by the assumption. Therefore $L$  is a $\cK({\sf f})$-moment functional by the implication (ii)$\to$(iv) of Haviland's Theorem 1.12.
$\hfill$   $\Box$
\smallskip

{\it Proof of Theorem \ref{schmps} (assuming Theorem \ref{mpschm} and Proposition \ref{prearchcom}):}\\
Suppose   $h\in \R_d[\ux]$ and $h(x)>0$ on $\cK({\sf f})$. Assume to the contrary that $h\notin T({\sf f})$. Since the preordering  $T({\sf f})$ is Archimedean by Proposition \ref{prearchcom},  Proposition \ref{eidelheit} applies, so  there  is a $T({\sf f})$-positive linear functional $L$ on $\R_d[\ux]$ such that $L(1)=1$ and $L(h)\leq 0$. By Theorem \ref{mpschm}, $L$ is a $\cK({\sf f})$-moment functional, that is, there is a measure $\mu\in M_+(\cK({\sf f}))$ such that $L(p)=\int_{\cK({\sf f})} p\, d\mu$  for $p\in \R_d[\ux]$. But $L(1)=\mu(\cK({\sf f}))=1$ and $h>0$ on $\cK({\sf f})$ imply that $L(h)>0$. This is a contradiction, since $L(h)\leq 0$. \hfill $\Box$

\medskip
The preordering $T({\sf f})$\, was defined as the sum of  sets $f_1^{e_1}\cdots f_k^{e_k}\cdot\sum \R_d[\ux]^2$. It is natural to ask whether or not  all such sets with mixed products $f_1^{e_1}\cdots f_k^{e_k}$ are really needed. To formulate the corresponding result we put $l_k:=2^{k-1}$ and let
$g_1,\dotsc,g_{l_k}$ denote the first $l_k$ polynomials of the following row of mixed  products:
\begin{align*}
    f_1,\dotsc,f_k,f_1f_2,f_1f_3,\dotsc,f_1f_k,\dotsc,f_{k-1}f_k,f_1f_2f_3,\dotsc, f_{k-2}f_{k-1}f_k,\dotsc,f_1f_2\dots,f_k.
\end{align*}
Let $Q({\sf g})$ denote the quadratic module generated by $g_1,\dotsc,g_{l_k}$, that is,
\begin{align*}Q({\sf g}):=\sum\R_d[\ux]^2+ g_1\sum\R_d[\ux]^2+\dots+g_{l_k}\sum\R_d[\ux]^2.\end{align*} 
The following  result of T. Jacobi and A. Prestel [JP] sharpens Theorem \ref{schmps}. 
\begin{theorem}\label{jacobiprestel}
    If the  set $\cK({\sf f})$ is compact and $h\in \R_d[\ux]$ satisfies $h(x)>0$ for all $x\in \cK({\sf f})$, then $h\in Q({\sf g})$.
\end{theorem}
We do not prove  Theorem \ref{jacobiprestel}; for a proof of this result we refer to [JP]. If we take Theorem \ref{jacobiprestel}  for granted and combine it  with Haviland's theorem 
1.12 we obtain the following corollary.
\begin{corollary}\label{corjacobiprestel}\index{Moment problem! for compact semi-algebraic sets}
    If  the set $\cK({\sf f})$ is compact and $L$ is a $Q({\sf g})$-positive linear functional on $\R_d[\ux]$, then $L$ is a $\cK({\sf f})$-moment functional.
\end{corollary}
We briefly discuss Theorem \ref{jacobiprestel}. If $k=1$, then  $Q({\sf f})=T({\sf f})$. However, for $k=2$,
\begin{align*}
    Q({\sf f})&=\sum \R_d[\ux]^2+f_1\sum \R_d[\ux]^2+f_2\sum \R_d[\ux]^2,
\end{align*}
so $Q({\sf f})$ differs from the preordering $T({\sf f})$ by the summand $f_1f_2\sum \R_d[\ux]^2$. If $k=3$, then
\begin{align*}
    Q({\sf f})
    %=\cS_{\sf f}
    =\sum \R_d[\ux]^2+&f_1\sum \R_d[\ux]^2+ f_2\sum \R_d[\ux]^2+f_3\sum \R_d[\ux]^2+f_1f_2\sum \R_d[\ux]^2 \,,
\end{align*} 
that is, the sets 
$g\sum \R_d[\ux]^2$ with $g=f_1f_3,f_2f_3,f_1f_2f_3$  do not enter into the definition of $Q({\sf f})$. For $k=4$,  no products of three or four generators appear in the definition of $Q({\sf f})$. For large $k$, only a small portion of mixed products occur in $Q({\sf f})$ and Theorem \ref{jacobiprestel} is an essential strengthening of Theorem \ref{schmps}.  
\smallskip

The next corollary 
characterizes in terms of moment functionals when a Radon measure on a compact semi-algebraic set  has a {\it bounded}   density\index{Moment functional! with bounded density} with respect to another Radon measure. A version for closed  sets is stated in Exercise 14.11 below.
\begin{corollary}\label{mpwithboundeddensity}
    Suppose that  the semi-algebraic set $\cK({\sf f})$ is compact. Let $\mu$ and $\nu$ be  finite Radon measures  on $\cK({\sf f})$ and let $L^\mu$  and $L^\nu$ be the corresponding moment functionals on $\R_d[\ux]$.
    There exists a function $\varphi \in L^\infty(\cK({\sf f}),\mu)$,  $\varphi(x)\geq 0$ $\mu$-a.e. on $\cK({\sf f})$, such that\, $d\nu=\varphi d\mu$\, if and only  if there is a constant $c> 0$ such that 
    \begin{align}\label{boundedesti}
        L^\nu(g)\leq cL^\mu(g)\quad \text{for}~~~ g\in T({\sf f}).
    \end{align}
\end{corollary}
\begin{proof}
    Choosing  $c\geq \|\varphi\|_{ L^\infty(\cK({\sf f}),\mu)}$, the  necessity of (\ref{boundedesti})  is easily verified.

    To prove  the converse we assume that (\ref{boundedesti}) holds. Then, by  (\ref{boundedesti}),\,  $L:=cL^\mu-L^\nu$  is a  $T({\sf f})$-positive linear functional  on $\R_d[\ux]$ and hence  a $\cK({\sf f})$-moment functional by Theorem \ref{mpschm}. Let $\tau$ be a representing measure of $L$, that is, $L=L^\tau$. Then we have $L^\tau+ L^\nu=cL^\mu$. Hence both $\tau+\nu$ and $c\mu$ are representing measures of the $\cK({\sf f})$-moment functional $cL^\mu$. Since $\cK({\sf f})$ is compact, $c\mu$ is determinate by Proposition \ref{detcomacpcase}, so that $\tau+\nu=c\mu$. In particular, this implies that $\nu$ is absolutely continuous with respect to $\mu$. Therefore, by the Radon--Nikodym theorem A.3, $d\nu =\varphi d\mu$ for some function $\varphi \in L^1(\cK({\sf f}),\mu)$,  $\varphi(x)\geq 0$ $\mu$-a.e. on $\cK({\sf f})$. Since $\tau+\nu=c\mu$, for each Borel subset $M$ of $\cK({\sf f})$ we have 
    \begin{align*}
        \tau (M)=c\mu(M)-\nu(M)=\int_M  (c-\varphi(x))d\mu\geq 0 .
    \end{align*}
    Therefore, $c-\varphi(x)\geq 0$~ $\mu$-a.e., so that $\varphi\in L^\infty(\cK({\sf f}),\mu)$ and $\|\varphi\|_{ L^\infty(\cK({\sf f}),\mu)}\leq c.$
    \end{proof}
We close this section 
by restating  Theorems  \ref{schmps} and \ref{mpschm} in the special case of compact real algebraic sets.
\begin{corollary}\label{compactalgvar}
    Suppose that $\cI$ is an ideal of $\R_d[\ux]$ such that the real algebraic set\, $V:=\cZ(\cI)=\{x\in \R^d:f(x)=0~~\text{for}~ f\in \cI\}$ is compact.
    \begin{itemize}
        \item[\em (i)]\,\, If $h\in \R_d[\ux]$ satisfies $h(x)>0$ for all $x\in V$, then $h\in \sum \R_d[\ux]^2 +\cI$.
        \item[\em (ii)]\,\,  If $p\in\R_d[\ux]/ \cI$ and $p(x)>0$ for all $x\in V$, then $p\in \sum (\R_d[\ux]/ \cI)^2$.
        \item[\em (iii)]~\, If $q\in \R[V]\equiv\R_d[\ux]/ \hat{\cI}$ and $q(x)>0$ for all $x\in V$, then $q\in \sum \R[V]^2$.
        \item[\em (iv)]~  Each positive linear functional  on $\R_d[\ux]$ which annihilates $\cI$ is a $V$-moment functional.
    \end{itemize}
\end{corollary}
\begin{proof}
    Put $f_1=1, f_2=h_1,f_3=-h_1,\dotsc, f_{2m}=h_m,f_{2m+1}=-h_m$, where  $h_1,\dotsc,h_m$ is a set of generators of $\cI$. Then, by (\ref{quqadraticzero}),
    the preordering $T({\sf f})$ is\,  
    $\sum \R_d[\ux]^2 +\cI$\, and the semi-algebraic set $\cK({\sf f})$ is $V=\cZ(\cI)$. Therefore,  Theorem  \ref{schmps} yields (i). 
    Since $\cI\subseteq \hat{\cI}$, (i) implies (ii) and (iii).  
    
    Clearly, a linear functional on $\R_d[\ux]$ is  $T({\sf f})$-positive if it is positive and annihilates $\cI$. Thus (iv) follows at once from Theorem \ref{mpschm}. 
    \end{proof}  

\begin{example} ({\it Moment problem on unit spheres})\index{Moment problem! for unit spheres}\\
    Let $S^{d-1}=\{x\in \R^d:x_1^2+\dots+x_d^2=1\}$\index[sym]{SAdA@$S^{d-1}$} 
    be the unit sphere of $\R^d$. Then $S^{d-1}$ is the real algebraic set $\cZ(\cI)$ for the ideal $\cI$ generated by $h_1(x)=x_1^2+\dots+x_d^2-1.$

    Suppose that $L$ is a linear functional  on $\R_d[\ux]$  such that
    \begin{align*}
        L(p^2)\geq 0 \quad {\rm and}~~ L((x_1^2+\dots+x_d^2-1)p)=0 \quad {\rm for }~~~ p\in \R_d[\ux].
    \end{align*} 
    Then it follows from  Corollary \ref{compactalgvar}(iv) that $L$ is an\, $S^{d-1}$-moment functional.
    
    Further, if $q\in \R[S^{d-1}]$  is strictly positive on $S^{d-1},$ that is, $q(x)>0$ for  $x\in  S^{d-1}$, then  $q\in \sum\R[S^{d-1}]^2$ by  Corollary \ref{compactalgvar}(iii). $\hfill \circ$
\end{example}

\section{The Archimedean Positivstellensatz  for quadratic modules and semirings}\label{reparchimodiules}

The main aim of this section is to  derive a representation theorem for Archimedean semirings and Archimedean quadratic modules (Theorem \ref{archpos}) and its application to the moment problem (Corollary  \ref{archrepmp}). 
By means of the so-called dagger cones we show that  to prove this general result it suffices to do so in the special cases of Archimedian semirings {\it or}  of Archimedean quadratic modules.
In this section we develop an approach based on semirings. At the end of Section \ref{Operator-theoreticappraochtothemomentprblem} we  give  a  proof using quadratic modules and Hilbert space operators.

Recall that $\sA$ is a {\it commutative real unital algebra}.
The {\it weak topology} on the dual $\sA^*$ is the locally convex topology  generated by the family of seminorms $f\to |f(a)|$, where $a\in {\sA}$. Then, for each $a\in {\sA}$, the function $a\to f(a)$ is continuous on ${\sA}^*$ in the weak topology. 
\begin{lemma}\label{kccompact}
    Suppose that $C$ is an  Archimedean unital cone of $\sA$. Then the set\, $\cK(C)=\{ \chi\in \hat{A}:\chi(a)\geq 0, a\in C\}$ is compact in the weak topology of $A^*$.
\end{lemma}
\begin{proof}
    Since $C$ is Archimedean, for any $a\in A$ there exists a number $\lambda_a>0$ such that $\lambda_a-a\in C$ and $\lambda_a+a\in C$. Hence for  $\chi\in \cK(C)$ we have $\chi(\lambda_a-a)\geq 0$ and $\chi(\lambda_a+a)\geq 0$, so that $\chi(a)\in [-\lambda_a,\lambda_a]$.  Thus there is an injection $\Phi$ of $\cK(C)$ into the topological product space $$P:=\prod\nolimits_{a\in A} ~[-\lambda_a,\lambda_a]$$ given by $\Phi(\chi)= (\chi(a))_{a\in A}$. From the definitions of the corresponding topologies it follows  that $\Phi$ is a homeomorphism of $\cK(C)$, equipped with the weak topology, on the subspace $\Phi(\cK(C))$ of $P$, equipped with the product topology.
    
    We show that the image $\Phi(\cK(C))$ is closed in  $P$. Indeed, suppose $(\Phi(\chi_i))_{i\in I}$ is a net from $\Phi(\cK(C))$ which converges to $\varphi=(\varphi_a)_{a\in a}\in P$.  Then, by the definition of the weak topology,   $\lim_i \Phi(\chi_i)(a)=\lim_i \chi_i(a)=\varphi_a$ for all $a\in A$. Since for each $i$ the map $a\mapsto \chi_i(a)$ is a  character  that is nonnegative on $\cK(C)$, so is $a\mapsto \varphi_a$. Hence there exists $\chi\in \cK(C)$ such that $\varphi_a=\chi(a)$ for $a\in A$. Thus, $\varphi =\Phi(\chi)\in \Phi(\cK(C)$.
    
    The product $P$ is  a compact topological space by Tychonoff's theorem. Hence its closed subset $\Phi(\cK(C))$ is also compact and so is $\cK(C)$, because $\Phi$ is a homeomorphism  of $\cK(C)$ and $\Phi(\cK(C))$.
\end{proof}

In our approach to the Archimedean Positivstellensatz we use
the following notion.
\begin{definition}
    For a unital convex cone $C$ in $\sA$  we define
    \begin{equation}
        C^\dagger =\{ a\in {\sA} : ~ a+\epsilon  \in C~~ \textup{ for~ all }~~ \epsilon \in {(0,+\infty)} \}.
        \label{eq:ddagger}
    \end{equation}
\end{definition}
Clearly, $C^\dagger$ is again a unital convex cone in $\sA$. Since $1\in C$, we have $C\subseteq C^\dagger$.
\begin{lemma}\label{daggersimple}
    For each unital convex cone $C$ in $\sA$, we have $\cK(C)=\cK(C^\dagger)$ and $(C^\dagger)^\dagger= C^\dagger$.
\end{lemma}
\begin{proof}
    It is obvious that $\cK(C^\dagger)\subseteq \cK(C)$, because $C\subseteq C^\dagger$. Conversely, let  $\chi\in \cK(C)$. If $a\in C^\dagger$, then $a+\epsilon \in C$ and hence $\chi(a+\epsilon)\geq 0$ for all $\varepsilon >0$. Letting $\varepsilon \searrow 0$, we get $\chi(a)\geq 0$. Thus $\chi\in \cK(C^\dagger)$. 
    
    Clearly, $C^\dagger \subseteq (C^\dagger)^\dagger$. To verify the converse, let $a\in (C^\dagger)^\dagger$. Then $a+\varepsilon_1 \in C^\dagger$ and $a+\varepsilon_1+\varepsilon_2 \in C$ for $\varepsilon_1>0$, $\varepsilon_2>0$, so $a+\varepsilon\in \C$ for all $\varepsilon >0$. Hence $a\in C^\dagger$.
\end{proof}
\begin{example}
    Let ${\sA}$ be a real algebra of bounded real-valued functions on a set $X$ which contains the constant functions. Then
    \begin{align*}
        C:=\{ f\in {\sA}: f(x)>0~~~  \textup{for~ all} ~~ x\in X\}
    \end{align*} is an Archimedean preordering of $\sA$ and 
    \begin{align}\label{exacdagger}
        C^\dagger=\{f\in {\sA}: f(x)\geq 0~~~ \textup{for~ all}~~x\in X\}.
    \end{align}
    We  verify formula (\ref{exacdagger}). If $f(x)\geq 0$ on $X$, then $f(x)+\varepsilon >0$ on $X$, hence $f+\varepsilon\in C$ for all $\varepsilon >0$, so that $f\in C^\dagger$. Conversely, if $f\in C^\dagger$, then $f+\varepsilon \in C$, hence $f(x)+\varepsilon >0$ on $X$ for all $\varepsilon >0$; letting $\varepsilon \searrow 0$, we get $f(x)\geq 0$ on $X$ . This proves (\ref{exacdagger}). 
\end{example}

\begin{proposition} \label{proposition:qmddagger}
    If $Q$ is an Archimedean quadratic module of ${\sA}$, then $Q^\dagger$ is an Archimedean preordering of ${\sA}$.
\end{proposition}
\begin{proof}
    Clearly, $Q^\dagger$ is a unital convex cone of ${\sA}$ that contains all squares. We only have to show that
    $Q^\dagger$ is closed under multiplication.
    
    Let $p,q\in Q$ and $\epsilon \in {(0,+\infty)}$ be given. We prove that $pq + \epsilon  \in Q$.
    Because $Q$ is Archimedean, there exists a $\lambda >0$ such that $\lambda  - p \in Q$.
    We recursively define a sequence $(r_k)_{k\in \N_0}$  of elements of ${\sA}$ by\, $r_0 := p/\lambda$
    and\, $r_{k+1} := 2 r_k - r_k^2$,  $k\in \N_0$. Then we have $pq-\lambda qr_0=0$ and $$pq - 2^{-(k+1)}\lambda q r_{k+1}=(pq-2^{-k}\lambda qr_k
    ) +2^{-(k+1)}\lambda qr_k^2.$$ 
    Therefore, since $q\in Q$ and $Q$ is  a quadratic module, it follows by induction that 
    \begin{align}\label{qrk}
        (pq - 2^{-k}\lambda q r_k) \in Q\, \quad {\rm for}~~k\in \N_0.
    \end{align} 
    Adding $2^{-(k+1)}\lambda (q+r_k)^2 \in Q$ we obtain\,
    $pq + 2^{-(k+1)}\lambda (q^2 + r_k^2) \in Q$\, for $k\in \N_0$. 
    %Further, using 
    %$2r_{k+1} = (r_k)^2(2-r_k) + (2-r_k)^2r_k \in Q
    % \quad\quad\textup{and}\quad\quad
    %\Unit - r_{k+1} = (\Unit-r_k)^2 \in Q$
    %\begin{equation*}
    %2 r_{k+1} = r_k)^2(2-r_k) + (2-r_k)^2r_k \in Q
    % \quad\quad\textup{and}\quad\quad
    % 1 - r_{k+1} = (1-r_k)^2 \in Q
    %\end{equation*}
    %one inductively finds that $r_k, \Unit-r_k \in Q$, hence\, $2\cdot\Unit - (r_k)^2 = r_{k+1} + 2(\Unit-r_k) \in Q$.
    For sufficiently large $k\in \N_0$ we have\, $\epsilon - 2^{-(k+1)}\lambda (q^2 + r_k^2) \in Q$
    because $Q$ is Archimedean. Adding\,
    $pq + 2^{-(k+1)}\lambda (q^2 + (r_k)^2) \in Q$ by (\ref{qrk})  yields $(pq+\epsilon) \in Q$.
    
    Now let $r,s \in Q^\dagger$ and  $\epsilon \in {(0,+\infty)}$.
    As $Q$ is Archimedean, there exists $\lambda >0$ such that $\lambda  - (r+s) \in Q$.\,
    Set $\delta := \sqrt{\lambda^2 + \epsilon}\, - \lambda$. 
    Since $r,s\in Q^\dagger$, we have $r+\delta , s+\delta  \in Q$ and $((r+\delta)(s+\delta)+ \delta\lambda)\in Q$, as shown in the preceding paragraph. Therefore, since $\delta^2 + 2\lambda \delta = \epsilon$, we obtain
    \begin{equation*}
        rs + \epsilon  = \big((r+\delta)(s+\delta ) + \delta \lambda \big) + \delta\big(\lambda  - (r+s) \big) \in Q .
    \end{equation*}
    Hence $rs\in Q^\dagger$.
\end{proof}
\begin{proposition}\label{archpre}
    Suppose that $S$ is an Archimedean semiring of $\sA$ and $C$ is an $S$-module. Then $C^\dagger$ is an Archimedean preordering of ${\sA} $
    and an $S^\dagger$-module. In particular, $S^\dagger$ is an Archimedean preordering.
\end{proposition}
\begin{proof}
    Let $a \in S^\dagger$ and $c\in C^\dagger$. Then, by definition,   $a + \delta  \in S$ and $c + \delta \in C$ for all $\delta >0.$ 
    Since $S$ is Archimedean, there exists a number $\lambda >0$ such that $\lambda  - a \in S\subseteq C$ and $\lambda - a \in S\subseteq C$.
    Given  $\epsilon \in {(0,+\infty)}$, we set 
    $\delta :=-\lambda + \sqrt{\lambda+\epsilon}$. Then $\delta > 0$ and $\delta^2 + 2\delta \lambda = \epsilon$, so we obtain
    \begin{equation*}
        ac + \epsilon 
        =
        (a + \delta)(c + \delta ) + \delta(\lambda-a)  + \delta(\lambda -c)
        \in
        C
        .
    \end{equation*}
    Therefore, $ac \in C^\dagger$. In particular, in the special case $C = S$ this shows that $S^\dagger$ is also a semiring. In the general case, it proves that $C^\dagger$  is an $S^\dagger$-module.

    Let $a\in {\sA}$. The crucial step is to prove that $a^2 \in S^\dagger$.
    For let $\varepsilon>0$. Since\, the polynomial  $x^2+\varepsilon$ is positive for all $x\in [-1,1]$, by Bernstein's theorem (Proposition 3.4) there exist numbers $m\in \N$ and $a_{kl}\geq 0$ for $k,l=0,\dotsc,m$ such that 
    \begin{align}\label{applbernstein}
        x^2+\varepsilon =\sum_{k,l=0}^m a_{kl} (1-x)^k(1+x)^l
    \end{align} 
    Since the semiring $S$ is Archimedean,
    there exists a  $\lambda >0$  such that $( \lambda + a)\in S$ and $(\lambda - a)\in S$. Then 
    $(1 + a / \lambda)\in S$ and $(1 - a/\lambda)\in S$ and hence $(1 + a / \lambda)^n\in S$ and 
    $(1 - a/\lambda)^n\in S$ for all $n\in \N_0$, because $S$ is a semiring. As usual, we 
    set $(1 \pm a/\lambda)^0=1$. Therefore, using (\ref{applbernstein}) and the fact that $S$ is closed under multiplication, we find
    \begin{equation*}
        (a/\lambda)^2 + \varepsilon
        =
        \sum_{k,l = 0}^m a_{kl} (1-(a/\lambda)^k (1+(a/\lambda)^l
        \in S.
    \end{equation*}
    Hence $(a^2+\lambda^2 \varepsilon)\in S$. Since $\lambda$ depends only on $a$ and $\varepsilon >0$ was arbitrary, this implies that
    $a^2 \in S^\dagger$.
    
    Thus, $S^\dagger$ is a semiring which contains all squares, that is, $S^\dagger$ is a preordering. 
    
    Since $S\subseteq C$ and hence $S^\dagger\subseteq C^\dagger$, $C^\dagger$ contains also all squares, so $C^\dagger$ is a quadratic module.
    Moreover, from $S\subseteq S^\dagger$ and $S \subseteq C \subseteq C^\dagger$ it follows that $C^\dagger$ and $S^\dagger$ are Archimedean because   $S$ is Archimedean by assumption. 
    
    Since  $C^\dagger$ is an Archimedean quadratic module as we have proved, 
    $(C^\dagger)^\dagger$ is an Archimedean preordering by  Proposition~\ref{proposition:qmddagger}. 
    By Lemma \ref{daggersimple}, $(C^\dagger)^\dagger = C^\dagger$.
\end{proof}
\begin{remark}
    For\, $\varepsilon =\frac{1}{k-1}, k\in \N,$~ there is the following explicit form of the identity (\ref{applbernstein}):
    \begin{equation*}
        x^2 + \frac{1}{k-1}=   \frac{1}{2^k k (k-1)}
        \sum_{\ell = 0}^{k} \binom{k}{\ell} (k-2\ell)^2 (1+x)^{k-\ell}(1-x)^{\ell}.
    \end{equation*}
\end{remark}

The following important result   is the  {\it Archimedean Positivstellensatz 
    for quadratic modules and    semirings}.
\index{Quadratic module! Archimedean! Archimedean Positivstellensatz}\index{Semiring! Archimedean! representation theorem} 

\begin{theorem}\label{archpos}
    Suppose that  $C$ is an $S$-module of an Archimedean semiring $S$ or 
    $C$ is an Archimedean quadratic module
    of the commutative unital real algebra ${\sA}$.
    For any $a\in {\sA}$, the following are equivalent:
    \begin{enumerate}
        \item[\em $(i)_C$] $\chi(a)>0$ for all $\chi\in\cK(C)$.
        \item[\em $(ii)_C$] There exists $\epsilon \in {(0,+\infty)}$ such that $a \in \epsilon  + C$.
    \end{enumerate}
\end{theorem}
The following simple fact is crucial for our proofs of Theorem \ref{archpos} given below.
\begin{lemma}\label{ccdagger}
    In the notation of Theorem \ref{archpos}, each of the conditions $(i)_C$ and $(ii)_C$ holds for $C$ if and only if it does for $C^\dagger$.
\end{lemma}
\begin{proof} Since $\cK(C)=\cK(C^\dagger)$ by Lemma \ref{daggersimple}, this is obvious of $(i)_C$. For $(ii)_C$, since $C\subseteq C^\dagger$, it suffices it verify that $(ii)_{C^\dagger}$ implies $(ii)_C$ . Indeed, if $a=2 \epsilon + c^\dagger$ with $\epsilon 
    >0$ and $c^\dagger \in C^\dagger$, then by the definition of $C^\dagger$ we have $c:=c^\dagger+\epsilon\in C$, so that $a=\epsilon +c\in C$. Thus, $(ii)_C$ is equivalent to $(ii)_{C^\dagger}$. 
    \end{proof}

Before proving the theorem, we discuss 
%some aspects of 
this result with a couple of remarks.
\begin{remark}\label{remarkarchpos}
    
    1.) First we emphasize that in strong contrast to   Theorem \ref{schmps}
    %in Section \ref{momentproblemstrictpos} (Theorem \ref{schmps}) 
    the above Theorem \ref{archpos} does not require that $\sA$ or $C$ or $S$ is finitely generated.
    
    2.) Using  the fact that the preordering $T({\sf f})$ is Archimedean (by Proposition \ref{prearchcom}) it is clear that Theorem\ref{schmps} follows directly from Theorem \ref{archpos}. In Section \ref{momentproblemstrictpos} we have given an ``elementary'' proof of Theorem \ref{schmps} which is based on 
    Proposition \ref{continuityLprop}(i) and does not depend on Theorem \ref{archpos}.

    3.) The proof of implication $(ii)_c\to (i)_C$ is very easy:
    Indeed, 
    if $a=\epsilon  +c$ with $c\in C$, then $\chi(a)=\epsilon\chi(1)+\chi(c)= \epsilon+\chi(c)\geq \epsilon >0$ for all $\chi\in \cK(C)$.
    
    4.) Since $1\in C$, $(ii)_C$  implies that $a\in C$. The stronger statement $a\in \epsilon +C$ is  given in order to get an equivalence of conditions $(i)_C$ and $(ii)_C$. 
    
    The main assertion of Theorem \ref{archpos} states that  {\it the positivity (!) of the values $\chi(a)$ for all $C$-positive characters on ${\sA}$ implies that  $a$ belongs to $C$}.

    5.)  Recall that $C^\dagger$ is an Archimedean preordering by Propositions \ref{proposition:qmddagger} and \ref{archpre}. Therefore, by Lemma \ref{ccdagger}, to prove Theorem \ref{archpos} it suffices to do so in the case when $C$ is an Archimedean preordering of ${\sA}$. In particular, it is enough to show Theorem \ref{archpos} for Archimedean semirings or for Archimedean quadratic modules. In this section we prove of Theorem \ref{archpos} for Archmimedean semirings, while in Section \ref{Operator-theoreticappraochtothemomentprblem} we give an approach for   Archimedean quadratic modules. 
\end{remark}

%\noindent
{\it Proof of Theorem \ref{archpos} for Archimedean semirings:}\\
The trivial implication $(ii)_C\to (i)_C$ was already noted in the preceding remark 3.).

We suppose  that $C$ is an Archimedean semirings of ${\sA}$ and prove the main implication $(i)_C\to (ii)_C$ . For let  $c\in {\sA}$ be such that $c\notin C$.
Then, by  Proposition \ref{eidelheit}, there exists an extremal (!) functional $\varphi$ of $C^\wedge$ such that $\varphi(1)=1$ and $\varphi(c)\leq 0$. We prove that $\varphi\in \hat{{\sA}}$, that is,
\begin{align}\label{characterprop}
    \varphi(ab)=\varphi(a)\varphi(b)\quad \textup{for}~~ a,b\in {\sA}.
\end{align}  

Let $a\in {\sA}$. Since $C$ is Archimedean, there exists  $\lambda>0$ such that $\lambda+a\in C$, so that $a=(\lambda + a)-\lambda\in C-C$. Thus, ${\sA}=C-C$. Hence  it suffices to verify (\ref{characterprop}) for $a\in C$ and similarly for $b\in C.$ Then $\varphi(a)\geq 0$, since $\varphi$ is $C$-positive.

Case 1:~ $\varphi (a)=0$.\\ Let $b\in C$ and choose $\lambda>0$ such that $\lambda -b\in C$. Then $(\lambda-b)a\in C$ and $ab\in C$   (because $C$ is a semiring!), so that    $\varphi((\lambda-b)a)= \lambda \varphi(a)-\varphi(ab)=-\varphi(ab)\geq 0$ and $\varphi (ab)\geq 0$. Hence $\varphi(ab)=0$, so that  (\ref{characterprop}) holds.

Case 2:~ $\varphi(a)>0$.\\ We choose $\lambda>0$ such that $(\lambda {-}a )\in C$ and $\varphi(\lambda{-}a)>0$. Because $C$ is a semiring, the functionals $\varphi_1(\cdot):=\varphi(a)^{-1}\varphi( a\cdot)$ and $\varphi_2(\cdot):=\varphi(\lambda -a)^{-1}\varphi((\lambda-a)\cdot)$ belong to the dual cone $C^\wedge$. They satisfy
\begin{align*}
    \varphi= \lambda^{-1}\varphi(a)\, \varphi_1 +  \lambda^{-1}\varphi(\lambda -a)\, \varphi_2 ,
\end{align*} 
so $\varphi$ is a convex combination of two  functionals from $C^\wedge$. Since $\varphi$ is extremal, it follows that $\varphi_1=\varphi$ which gives  (\ref{characterprop}). 

Summarizing both cases, we have shown that $\varphi\in \hat{{\sA}}$.
Recall that $\varphi(c)\leq 0.$

Now it is easy to prove that $(i)_C$ implies $(ii)_C$. Let $a\in {\sA}$ be as in $(i)_C$. Then, since the function $a\to \varphi (a)$ is continuous on the compact set $\cK(C)$  in the weak topology (by Lemma \ref{kccompact}), there exists  $\epsilon>0$ such that $c:=a-\epsilon$ also satisfies $\varphi(c)>0$ for all $\varphi\in \cK(C)$. Therefore, by the preceding proof, $c\notin C$ cannot hold, so that $c\in C$. Hence $a=\epsilon + c\in \epsilon +C.$
$\hfill \Box$ 

\begin{corollary} 
    Under the assumptions of Theorem \ref{archpos}, we have
    \begin{align*}
        C^\dagger =\{a\in {\sA}: \chi(a)\geq 0 ~ \textup{ for all } ~\chi\in \cK(C)\, \}.
    \end{align*}
\end{corollary}
\begin{proof}
    If $\chi(a)\geq 0$ for  $\chi\in \cK(C)$, then  for  $\epsilon>0$ we have\,
    $\chi(a+\epsilon )=\chi(a)+\epsilon >0$.  Therefore,
    $a+\epsilon\in C$  by Theorem \ref{archpos}, so that $a\in C^\dagger $. 
    
    Conversely, if $a\in C^\dagger$ and $\chi\in \cK(C)$, then $a+\epsilon \in C$. Hence $\chi(a)+\epsilon=\chi(a+\epsilon ) \geq 0$  for all $\epsilon>0$. Letting 
    $\epsilon\searrow 0$ yields $\chi(a)\geq 0$.   
    \end{proof}

The following is the main application of Theorem \ref{archpos} to the moment problem.
\begin{corollary}\label{archrepmp}
    Retain the assumptions  of  Theorem \ref{archpos}. Suppose that $L$ is a linear functional on ${\sA}$ such that $L(c)\geq 0$ for all $a\in C$. Then there exists a Radon measure $\mu$ on the compact topological space $\cK(C)$ such that
    \begin{align}
        L(a)=\int_{\cK(C)} \chi(a)~ d\mu(\chi)~~~  \textup{for}~~ a\in {\sA}.
    \end{align}
\end{corollary}
\begin{proof} Let $a\in {\sA}$ be such that $\chi(a)\geq 0$ for $\chi\in \cK(C)$. Then, for each $\epsilon >0$, $a+\epsilon$ satisfies $(i)_C$, so $a+\epsilon \in C$ by Theorem \ref{archpos}. Hence $L(a+\epsilon )=L(a)+\epsilon L(1)\geq 0.$ Letting $\epsilon\searrow 0$, we get $L(a)\geq 0$. Now the assertion follows from Proposition 1.9. 
    \end{proof}
\section{The Archimedean representation theorem for polynomial algebras}\label{archimedeanpolynomials}

In this  section we first restate Theorem \ref{archpos} and Corollary \ref{archrepmp} in the special case when ${\sA}$ is the polynomial algebra $\R_d[\ux]$.
%Then we use the Archimdean Positivstellensatz for semirings to derive a sharpening of Theorem ...

We begin with the case of Archimedean quadratic modules. Assertion (i) of the following theorem is also called the {\it Archimedean Positivstellensatz}.\index{Positivstellensatz! Archimedean}
\begin{theorem}\label{archmedpsq}
    Let ${\sf f}=\{f_1,\dotsc,f_k\}$ be a finite subset of $\R_d[\ux]$. Suppose that the quadratic module $Q({\sf f})$ defined by (\ref{quadraticqf}) is Archimedean. 
    \begin{itemize}
        \item[\em (i)]\, If $h\in \R_d[\ux]$ satisfies $f(x)> 0$ for all $x\in \cK({\sf f})$, then $h\in Q({\sf f}).$
        \item[\em (ii)]~ 
        Any $Q({\sf f})$-positive linear functional $L$ on\, $\R_d[\ux]$ is a $\cK({\sf f})$-moment functional, that is, there exists a measure $\mu\in M_+(\R^d)$ supported on the compact set $\cK({\sf f})$ such that\, $L(f)=\int f(x)\, d\mu(x)$\, for\,  $f\in \R_d[\ux]$.\index{Moment problem! for compact semi-algebraic sets}
    \end{itemize}
\end{theorem}
\begin{proof}
    Set ${\sA}=\R_d[\ux]$ and $C=Q({\sf f})$. As noted in Example \ref{ckcpoly},  characters $\chi$ of ${\sA}$ correspond to  points $\chi_t\cong t$ of $\R^d$ 
    and we have $\cK(Q)=\cK({\sf f})$ under this identification. Hence the assertions of (i) and (ii) follow at  once from Theorem \ref{archpos} and Corollary \ref{archrepmp}, respectively. 
    \end{proof}  

Next we turn to modules for semirings.
\begin{example}\label{exacfg}
    Let ${\sf f}=\{f_1,\dotsc,f_k\}$ and ${\sf g}=\{g_0=1,g_1,\dotsc,g_r\}$ be finite subsets of $\R_d[\ux]$, where $k\in \N, r\in \N_0$. Then
    \begin{align}\label{cfg}
        C({\sf f},{\sf g}):= g_0 S({\sf f})+ g_1 S({\sf f})+\cdots+ g_r S({\sf f})
    \end{align}
    is an\, $S({\sf f})$-module for the semiring\, $S({\sf f})$. Clearly, 
    $\cK(C({\sf f},{\sf g}))= \cK({\sf f})\cap \cK({\sf g})$.
    
    Note that in the special case $r=0$\, the $S({\sf f})$-module $C({\sf f}, {\sf g})$ is just the semiring $S({\sf f})$ itself and  $\cK(C({\sf f},{\sf g}))= \cK({\sf f}).$
\end{example}

\begin{theorem}\label{archmedps}
    Let ${\sf f}=\{f_1,\dotsc,f_k\}$ and ${\sf g}=\{g_0=1,g_1,\dotsc,g_r\}$ be  subsets of $\R_d[\ux]$, where $k\in \N, r\in \N_0$. Suppose that the semiring  $S({\sf f})$ defined by (\ref{semiringf}) is Archimedean. Let $C({\sf f},{\sf g})$ denote the $S({\sf f})$-module defined by (\ref{cfg}).
    \begin{itemize}
        \item[\em (i)]\, If $h\in \R_d[\ux]$ satisfies $h(x)> 0$ for all $x\in \cK({\sf f})\cap\cK({\sf g})$, then $h\in C({\sf f}, {\sf g}).$
        \item[\em (ii)]~
        Suppose $L$ is a linear functional  on\, $\R_d[\ux]$ such that $L(f)\geq 0$ for all $f\in  C({\sf f}, {\sf g})$. Then $L$ is a 
        $ \cK({\sf f})\cap \cK( {\sf g})$--moment functional, that is, there is a measure $\mu\in M_+(\R^d)$ supported on the compact semi-algebraic set $\cK({\sf f})\cap \cK({\sf g})$ such that\, $L(f)=\int f(x)\, d\mu(x)$\, for all\,  $f\in \R_d[\ux]$.\index{Moment problem! for compact semi-algebraic sets}
    \end{itemize}
\end{theorem}
\begin{proof}
    Combine Theorem \ref{archpos} and Corollary \ref{archrepmp} with Example \ref{cfg}.
    \end{proof} 
\smallskip

If $r=0$, then the $S({\sf f})$-module $C({\sf f}, {\sf g})$ coincides with  the semiring $S({\sf f})$  and we have $\cK(C({\sf f},{\sf g}))= \cK({\sf f}).$ Then  Theorem \ref{archmedps}(i) is  the Archimedean Positivstellensatz for semirings in the special case of the polynomial algebra $\R_d[\ux]$.\smallskip

The next theorem is an application of Theorem \ref{archmedps}. It sharpens  Theorem \ref{schmps} by representing positive polynomials on a compact semi-algebraic set by  a certain subset of the corresponding preordering.

\begin{theorem}\label{auxsemiring}
    Suppose ${\sf f}=\{f_1,\dotsc,f_r\}$, $r\in \N$, is a subset of $\R_d[\ux]$ such that the semi\-algebraic set 
    $\cK({\sf f})$
    is compact.   
    Then there  exist polynomials $p_1,\dotsc,p_s\in \R_d[\ux]$, $ s\in \N,$ such that
    the semiring $S$ of $\R_d[\ux]$ generated by $f_1,\dotsc,f_r, p_1^2, \dots,p_s^2$ is Archimedean. 
    
    If $h\in \R_d[\ux]$
    satisfies $h(x)>0$ for all  $x\in \cK({\sf f})$, then $h$ is a finite sum of polynomials 
    \begin{align}\label{specialform}
        \alpha\, f_1^{e_1}\cdots f_r^{e_r}~ f_1^{2n_1}\cdots f_r^{2n_r}\, p_1^{2k_1}\cdots p_s^{2k_s},
    \end{align}
    where $\alpha\geq 0$, $e_1,\dotsc,e_r\in \{0,1\}$, $n_1,\dotsc,n_r, k_1,\dotsc,k_s\in \N_0$.
    
    Further, each linear functional on $\R_d[\ux]$ that is nonnegative on all polynomials \eqref{specialform}
    (with $\alpha=1$) is a $\cK({\sf f})$-moment functional.
\end{theorem}
\begin{proof}
    Since  the set $\cK({\sf f})$ is compact,  there are numbers $\alpha_j>0, \beta_j>0$ such that 
    \begin{equation}\label{poshj}
        \alpha_j+x_j>0~~\textup{and}~~ \beta_j-x_j>0~~ \textup{for}~~ x\in \cK(f_1,\dotsc,f_r),~j=1,\dotsc,d.
    \end{equation} 
    Therefore, by Theorem \ref{schmps}, the polynomials $\alpha_j+x_j>0,\beta_j-x_j>0$ are in the preordering $T(f_1,\dotsc,f_r)$.
    By the definition (\ref{preorderingtf}) of  $T(f_1,\dotsc,f_r)$, this means that each polynomial $\alpha_j+ x_j$,  $\beta_j-x_j$ is a finite sum of  polynomials  of the form $f_1^{e_1}\cdots f_r^{e_r}p^2$ with $p\in \R_d[\ux]$ and $e_1,\dotsc,e_r\in \{0,1\}$. Let $S$
    denote the semiring generated by  $f_1,\dotsc,f_r$ and all squares $p^2$ occurring
    in these representations  of the   polynomials 
    
    $\alpha_j+ x_j,\beta_1-x_j$, where $j=1,\dotsc,d$. Then, by construction,   $x_1,\dotsc,x_d$ belong to $\R_d[\ux]_b(S)$, so  $S$ is Archimedean by  Lemma \ref{boundedele2}. Since $f_1,\dotsc,f_r\in S$, $\cK(S)$ is the set of point evaluations at $\cK(f_1,\dotsc,f_r)$.
    
    By its construction, the semiring $S$ defined above is generated by polynomials $f_1,\dotsc,f_r$,  $p_1^2,\dotsc,p_s^2$.  The Archimedean Positivstellensatz
    for semirings (Theorem \ref{archpos} or  Theorem \ref{archmedps})  yields $h\in S$. This means that $h$ is  a finite sum of
    terms \eqref{specialform}. 
    By Haviland's theorem (Theorem 1.12) this implies the last assertion.
    \end{proof}

In  the above proof  the polynomials $x_1,\dotsc,x_d$ can be replaced by any finite set of algebra generators of $\R_d[\ux].$ Note that (\ref{poshj}) means that the  set $\cK({\sf f})$ is contained in the $d$-dimensional rectangle $[-\alpha_1,\beta_1]\times \dots \times [-\alpha_d,\beta_d]$.

We illustrate the preceding result with  an example. 

\begin{example}\label{module}
    Let $S$ denote the semiring of $\R_d[\ux]$ generating by the  polynomials
    \begin{align}
        f(x) := 1-x_1^2-\cdots-x_d^2,~~
        g_{j,\pm}(x) := (1\pm x_j)^2,\;
        j=1,\dotsc,d.
    \end{align}
    Obviously, $\cK(S)$ is  the closed unit ball
    \begin{equation*}
        \cK(f)=\{ x\in \R^d:~ x_1^2+\cdots+x_d^2\leq 1\}.
    \end{equation*}
    Then, since 
    \begin{equation*}
        d+1\pm 2x_k=(1-x_1^2-\cdots-x_d^2) +(1\pm x_k)^2+\frac{1}{2} \sum_{i=1, i\neq k}^d \big((1+x_j)^2+(1-x_j)^2\big)\in S,
    \end{equation*}
    for  $k=1,\dotsc,d$,\, Lemma \ref{boundedele2} implies that   $S$ is Archimedean. Therefore, by Theorem
    \ref{archpos} (or  Theorem \ref{archmedps}), each polynomial $h\in \R_d[\ux]$ that is positive in all points of  the closed unit ball $\cK(f)$ belongs to $S$. This means that $h$ is of the form
    %This means that $h$ is a finite sum of terms
    %\begin{align*}
    % &\alpha~ f^{2n}(1-x_1)^{2k_1}(1+x_1)^{2\ell_1}\cdots (1-x_d)^{2k_d}(1+x_d)^{2\ell_d}, \\&\beta~ f\, f^{2n}(1-x_1)^{2k_1}(1+x_1)^{2\ell_1}\cdots %(1-x_d)^{2k_d}(1+x_d)^{2\ell_d},
    %\end{align*} 
    %where\, $\alpha \geq 0$, $\beta\geq 0$, $n,k_i,\ell_i\in \N_0$. 
    \begin{multline*}
        h(x)
        =\sum_{n,k_i,\ell_i=0}^m \alpha_{n,k_1,\ell_1,\dotsc,k_d,\ell_d}f^{2n}(1-x_1)^{2k_1}(1+x_1)^{2\ell_1}\dotsm (1-x_d)^{2k_d}(1+x_d)^{2\ell_d}\\
        + f\sum_{n,k_i,\ell_i=0}^m \beta_{n,k_1,\ell_1,\dotsc,k_d,\ell_d} f^{2n}(1-x_1)^{2k_1}(1+x_1)^{2\ell_1}\dotsm (1-x_d)^{2k_d}(1+x_d)^{2\ell_d},
    \end{multline*} 
    where\, $m\in \N_0$ and  $\alpha_{n,k_1,\ell_1,\dotsc,k_d,\ell_d}\geq 0,~ \beta_{n,k_1,\ell_1,\dotsc,k_d,\ell_d}\geq 0$. 
    This formula is a distinguished weighted sum of squares representation of the positive polynomial $h$. 
    
    The Archimedean Positivstellensatz for quadratic modules  (Theorem \ref{archmedpsq}) gives in this case the weaker assertion \,$h(x)=\sigma_1+ f\sigma_2$, with $\sigma_1,\sigma_2\in \sum \R_d[\ux]^2.$
\end{example}

\section{The operator-theoretic approach to the moment problem}\label{Operator-theoreticappraochtothemomentprblem}
The spectral theory of self-adjoint operators in Hilbert space is
well suited to the moment problem and provides  powerful techniques for the study of 
this problem.
The technical tool that relates the multidimensional moment problem to Hilbert space operator
theory  is the {\it Gelfand--Naimark--Segal construction},  briefly the {\it GNS-construction}. We 
develop this construction first for a general $*$-algebra (see  [Sm4, Section 8.6] or [Sm20, Section 4.4]]
and then we specialize to the polynomial algebra.

Suppose that ${\sA}$ is a unital (real or complex) $*$-algebra. Let $\K=\R$ or $\K=C$.
\begin{definition} 
    Let $(\cD,\langle \cdot, \cdot \rangle)$ be a unitary space. A \emph{$*$-representation}\index{Representation@$*$-Representation}\index{Representation@$*$-Representation! domain}\index{Representation@$*$-Representation! algebraically cyclic vector} of ${\sA}$ on   $(\cD,\langle \cdot, \cdot \rangle)$ is an algebra homomorphism $\pi$ of ${\sA}$ into the algebra $L(\cD)$ of linear operators mapping $\cD$ into itself such that $\pi(1)\varphi=\varphi$ for $\varphi\in \cD$ and
    \begin{align}\label{defstarrep}
        \langle \pi(a)\varphi,\psi \rangle =\langle \varphi ,\pi(a^*)\psi\rangle \quad {\rm for}\quad a\in {\sA} ,~~\varphi, \psi \in \cD. 
    \end{align}
    The unitary space $\cD$ is called the \emph{domain} of\, $\pi$ and denoted by\ $\cD(\pi)$.
    A vector $\varphi\in \cD$ is called \emph{algebraically cyclic}, briefly \emph{a-cyclic}, for $\pi$ if\, $\cD=\pi({\sA})\varphi$.
\end{definition}
Suppose that  $L$ is a positive linear functional on ${\sA}$, that is, $L$ is a linear functional  such that $L(a^*a)\geq 0$ for  $a\in {\sA}$. Then, by Lemma 2.3,    the Cauchy--Schwarz inequality holds:
\begin{align}\label{cauchyschwarzineq}
    |L(a^*b)|^2 \leq
    L(a^*a)L(b^*b)\quad {\rm for}\quad  a,b \in {\sA}.
\end{align}
\begin{lemma}\label{idealN_L}
    $\cN_L:=\{ a\in {\sA}: L(a^*a)=0\}$\index[sym]{NCL@$\cN_L$} is a left ideal of the algebra ${\sA}$. 
\end{lemma}
\begin{proof}
    Let $a,b\in \cN_L$ and $x\in {\sA}$. Using (\ref{cauchyschwarzineq}) we obtain 
    \begin{align*}
        |L((xa)^*xa)|^2 =|L((x^*xa)^*a)|^2\leq L((x^*xa)^*x^*xa)L(a^*a)=0,
    \end{align*}
    so that $xa\in \cN_L$. Applying again  (\ref{cauchyschwarzineq})\ we  get $L(a^*b)=L(b^*a)=0$. Hence
    \begin{align*}
        L((a+b)^*(a+b)) =L(a^*a) +L(b^*b)+ L(a^*b)+L(b^*a)=0,
    \end{align*}
    so that $a+b\in \cN_L$. Obviously, $\lambda a\in \cN_L$ for $\lambda\in \K$.
    \end{proof}  

Hence there exist a well-defined scalar
product $\langle \cdot,\cdot\rangle_L$ on the quotient vector space
$\cD_L{=}{\sA}/ \cN_L$ and a well-defined algebra homomorphism $\pi_L:
{\sA}{\to} L(\cD_L)$ given by
\begin{align}\label{defscalarGNS}
    \langle
    a+\cN_L,b+\cN_L\rangle_L =L(b^*a) ~~{\rm and}~~
    \pi_L(a)(b+\cN_L)=ab+\cN_L,~~a,b \in {\sA}.
\end{align} Let
$\cH_L$ denote the Hilbert space completion of the pre-Hilbert
space $\cD_L$. If no confusion can arise 
we write $\langle \cdot,\cdot\rangle$ for $\langle \cdot,\cdot\rangle_L$ and $a$ for $a+\cN_L$.  Then we have $\pi_L(a)b=ab$, in particular $\pi_L(1)a=a$, and
\begin{align}\label{Lrep}
    \langle \pi_L(a)b,c\rangle = L(c^*ab)= L((a^*c)^*b)=\langle b,\pi_L(a^*)c
    \rangle \quad{\rm for}\quad a,b,c \in {\sA}.
\end{align}
Clearly, $\cD_L=\pi_L({\sA})1$. Thus, we have shown that {\it $\pi_L$\index[sym]{pGL@$\pi_L$}\index[sym]{DCL@$\cD_L$} is a $*$-representation of ${\sA}$ on the domain $\cD(\pi_L)=\cD_L$ and  $1$ is an $\rm{a}$-cyclic vector for\, $\pi_L$}.  Further, we have
\begin{align}\label{gnslfunctional}
    L(a)=\langle \pi_L(a)1,1 \rangle \quad {\rm for} \quad a\in {\sA}.
\end{align}
\begin{definition}\label{defgnsrepL}
    $\pi_L$ is called the\, {\em{GNS-representation}} of ${\sA}$ associated with $L$.\index{Gelfand--Naimark--Segal construction, GNS-construction}
\end{definition}

We show that the GNS-representation  is unique up to unitary equivalence.  Let $\pi$ be another $*$-representation of ${\sA}$ with  a-cyclic vector $\varphi\in \cD(\pi)$ on a dense domain $\cD(\pi)$ of a Hilbert space $\cG$ such that $L(a)=\langle \pi(a)\varphi, \varphi\rangle$ for all $a\in {\sA}$.  For $a\in {\sA}$,
\begin{align*}
    \|\pi(a)\varphi\|^2=\langle\pi(a)\varphi,\pi(a)\varphi\rangle=\langle \pi(a^*a)\varphi,\varphi\rangle=L(a^*a)
\end{align*}
and similarly  $\|\pi_L(a)1\|^2=L(a^*a)$. Hence there is an   isometric linear map $U$ given by\, $U(\pi(a)\varphi)=\pi_L(a)1, a\in {\sA},$\,  of $\cD(\pi)=\pi({\sA})\varphi$\, onto\, $\cD(\pi_L)=\pi_L({\sA})1$. Since the domains $\cD(\pi)$ and $\cD(\pi_L)$ are dense in $\cG$ and $\cH_L$, respectively,  $U$  extends by continuity to a unitary operator of $\cG$ onto $\cH_L$. For $a,b\in{\sA}$ we derive
\begin{align*}
    U\pi(a)U^{-1}(\pi_L(b)1)=U\pi(a)\pi(b)\varphi =U\pi(ab)\varphi=\pi_L(ab)1=\pi_L(a)(\pi_L(b)1),
\end{align*}
that is,\, $U\pi(a)U^{-1}\varphi=\pi_L(a)\varphi$\,  for $\varphi \in \cD(\pi_L)$ and $a\in {\sA}$. By definition, this means that the $*$-representations $\pi$ and $\pi_L$ are unitarily equivalent.
\smallskip

Now we specialize the preceding to the $*$-algebra
$\C_d[\ux]\equiv \C[x_1,\dotsc,x_d]$ with involution determined by $(x_j)^*:=x_j$ for $j=1,\dotsc, d$. 

Suppose that $L$ is a positive linear functional on $\C_d[\ux]$. Since $(x_j)^*=x_j$, it follows from (\ref{Lrep}) that $X_j:=\pi_L(x_j)$ is a symmetric operator on the domain $\cD_L$. The operators $X_j$ and $X_k$  commute (because  $x_j$ and $x_k$ commute in $\C_d[\ux]$) and $X_j$ leaves the domain $\cD_L$ invariant (because $x_j\C_d[\ux]\subseteq \C_d[\ux]$). That is,  $(X_1,\dotsc,X_d)$ is a $d$-tuple of {\it pairwise commuting symmetric
    operators acting on the dense invariant domain}\, $\cD_L=\pi_L(\C_d[\ux])1$\, of the
Hilbert space $\cH_L$. Note that this $d$-tuple $(X_1,\dotsc,X_d)$ essentially depends on the given positive linear functional $L$.

The next theorem is the crucial result of the operator approach to the multidimensional moment problem and it is the counterpart of Theorem 6.1.

. It relates solutions of the moment problem to spectral
measures of strongly commuting $d$-tuples $(A_1,\dotsc,A_d)$\,  of self-adjoint
operators which extend our given  $d$-tuple $(X_1,\dotsc,X_d)$. 
\begin{theorem}\label{mp-spec}
    A positive linear functional $L$ on the $*$-algebra $\C_d[\ux]$ is a moment functional if and only if there exists a $d$-tuple $(A_1,\dotsc,A_d)$ of strongly commuting self-adjoint operators $A_1,\dotsc,A_d$ acting on a Hilbert space $\cK$ such that\, $\cH_L$ is a subspace of\, $\cK$ and $X_1\subseteq A_1,\dotsc, X_d \subseteq A_d$. If this is fulfilled and\, $E_{(A_1,\dotsc,A_d)}$\, denotes the spectral
    measure of the $d$-tuple\, $(A_1,\dotsc,A_d)$, then\, $\mu(\cdot)= \langle
    E_{(A_1,\dotsc,A_d)}(\cdot)1,1 \rangle_\cK$\, is a solution of the moment problem for $L$.
    
    Each solution of the moment problem for $L$ is of this form.
\end{theorem}

First we explain the notions occurring in this theorem (see  [Sm9, Chapter 5]  for the corresponding results and more details). 

A $d$-tuple $(A_1,\dotsc,A_d)$ of self-adjoint operators $A_1,\dotsc,A_d$ acting on a Hilbert
space $\cK$ is called {\it strongly commuting} if for all $k,l=1,\dotsc,d, k\neq l,$ the resolvents
$(A_k-{\ii} I)^{-1}$ and $(A_l-{\ii} I)^{-1}$ commute, or equivalently, 
the  spectral measures $E_{A_k}$ and $E_{A_l}$
commute (that is,  $E_{A_k}(M)E_{A_l}(N)=E_{A_l}(N)E_{A_k}(M)$ for
all Borel subsets $M,N$ of $\R$). (If the self-adjoint operators  are bounded,  strong commutativity and ``usual" commutativity are equivalent.) The  spectral theorem\index{Spectral theorem! for strongly commuting self-adjoint operators}  states that, for such a $d$-tuple, there exists a unique spectral measure
$E_{(A_1,\dotsc,A_d)}$ on the Borel $\sigma$-algebra of $\R^d$ such that 
\begin{align*}
    A_j=\int_{\R^d} \lambda_j~ dE_{(A_1,\dotsc,A_d)}(\lambda_1,\dotsc,\lambda_d),~j=1,\dotsc,d.
\end{align*}
The spectral measure $E_{(A_1,\dotsc,A_d)}$ is the product of spectral measures $E_{A_1},\cdots E_{A_d}$. Therefore, if $M_1,\dotsc,M_d$ are Borel subsets of $\R$, then 
\begin{align}
    \label{jointspec}
    E_{(A_1,\dotsc,A_d)}(M_1\times\cdots \times M_d)
    =E_{A_1}(M_1)\cdots E_{A_d}(M_d).
\end{align}

\noindent {\it Proof of Theorem \ref{mp-spec}:}

First assume that $L$ is the moment functional and let $\mu$ be  a representing measure of $L$.
It is well-known and easily checked by the preceding remarks that the multiplication operators $A_k$,
$k=1,\dotsc,d$, by the coordinate functions $x_k$ form a $d$-tuple
of strongly commuting self-adjoint operators on the Hilbert space
$\cK:=L^2(\R^d,\mu)$  such that $\cH_L \subseteq \cK$ and $X_k\subseteq
A_k$ for $k=1,\dotsc,d$.
The spectral measure $E:= E_{(A_1,\dotsc,A_d)}$ of this $d$-tuple acts by\, $E(M)f=\chi_M \cdot f$, $f \in L^2(\R^d,\mu)$, where $\chi_M$ is the characteristic function of the Borel set $M\subseteq \R^d$. This implies that  $\langle E(M)1,1 \rangle_\cK =\mu(M)$. Thus, $\mu(\cdot)=\langle E( \cdot ) 1,1\rangle_\cK$.

Conversely, suppose that $(A_1,\dotsc,A_d)$ is such a $d$-tuple. By the multidimensional spectral theorem [Sm9, Theorem 5.23] this $d$-tuple has  a joint  spectral measure
$E_{(A_1,\dotsc,A_d)}$. Put
$\mu(\cdot):=\langle E_{(A_1,\dotsc,A_d)}(\cdot)1,1 \rangle_\cK$. Let $p \in \C_d[\ux]$. Since
$X_k \subseteq A_k$, we have \begin{align*}p(X_1,\dotsc,X_d) \subseteq
    p(A_1,\dotsc,A_d).\end{align*} Therefore, since the polynomial $1$ belongs to the domain of
$p(X_1,\dotsc,X_d)$, it is also in the domain of $p(A_1,\dotsc,A_d)$. Then
\begin{align*}
    \int_{\R^d}& p(\lambda)~ d\mu(\lambda) = \int_{\R^d} p(\lambda)~ d\langle
    E_{(A_1,\dotsc,A_d)}(\lambda)1,1 \rangle_\cK  =\langle p(A_1,\dotsc,A_d) 1,1 \rangle_\cK  \\&
    = \langle p(X_1,\dotsc,X_d)1,1 \rangle  = \langle \pi_L(p(x_1,\dotsc,x_d))1,1 \rangle =L(p(x_1,\dotsc,x_d)),
\end{align*}
where the second equality follows from the functional calculus and the last from (\ref{gnslfunctional}). This shows that
$\mu$ is a solution of the moment problem for $L$. $\hfill$ $\Box$
\smallskip

\begin{proposition}\label{archmiboundedop}
    Suppose  $Q$ is an Archimedean quadratic module\index{Quadratic module! Archimedean} of a commutative real unital algebra ${\sA}$. Let $L_0$ be a $Q$-positive $\R$-linear functional  on  ${\sA}$ and let $\pi_L$ be the  GNS representation of its extension $L$ to a $\C$-linear functional  on the complexification ${\sA}_\C={\sA}+\ii {\sA}$.  Then all operators $\pi_L(a)$, $a\in {\sA}_\C$, are bounded.
\end{proposition}
\begin{proof} 
    Since $\sum ({\sA}_\C)^2=\sum {\sA}^2$ by Lemma 2.17(ii) and $\sum {\sA}^2\subseteq Q$, $L$ is a positive linear functional on ${\sA}_\C$, so the GNS representation $\pi_L$ is well-defined.
    
    It  suffices to prove that $\pi_L(a)$ is bounded for $a\in {\sA}$. Since $Q$ is Archimedean,
    $\lambda -a^2\in Q$ for some $\lambda >0$. Let $x\in {\sA}_\C$.   By Lemma 2.17(ii),  $ x^*x(\lambda-  a^2)\in Q$ and hence\,  $L(x^*xa^2)=L_0(x^*xa^2)\leq \lambda L_0(x^*x)=\lambda L(x^*x)$, since $L_0$ is $Q$-positive. Then
    \begin{align*}
        \|\pi_L(a)\pi_L(x)1\|^2&= \langle \pi_L(a)\pi_L(x)1,\pi_L(a)\pi_L(x)1\rangle=\langle\pi_L((ax)^*ax)1,1\rangle\\ &= 
        L((ax)^*ax)=L(x^*xa^2)\leq \lambda L(x^*x)=\lambda \|\pi_L(x)1\|^2,
    \end{align*}
    where we used (\ref{defstarrep}) and (\ref{gnslfunctional}). That is, $\pi_L(a)$ is bounded on $\cD(\pi_L)$.
    \end{proof}  

We now illustrate the power of the operator approach to moment problems by giving short proofs of Theorems
\ref{archpos} and \ref{archmedps}. 

From remark \ref{remarkarchpos}, 6.), we recall that  in order to prove Theorem \ref{archpos} in the general case it suffices to do this in the special case  when $C$ is an Archimedean semiring {\it or} when $C$ is an Archimedean quadratic module. In Section \ref{reparchimodiules} we  have given an approach based on semirings.  Here we prove it for quadratic modules. 
\smallskip

{\it Proof of Theorem \ref{archpos} for Archimedean quadratic modules:}

Suppose that $C$ is an Archimedean quadratic module of $\sA$. As in the proof for semirings, the implication $(ii)_C\to (i)_C$ is trivial and it suffices to prove that $(i)_C$ implies $a\in C$ (otherwise replace $a$ by $a-\varepsilon$ for small $\varepsilon>0$.). 

Assume to the contrary that $a$ satisfies $(i)_C$, but $a\notin C$. Since $C$ is Archimedean, by Proposition \ref{eidelheit} there is a $C$-positive $\R$-linear functional $L_0$ on  ${\sA}$ such that $L_0(1)=1$ and $L_0(a)\leq 0$. Let $\pi_L$ be the  GNS representation of its extension to  a $\C$-linear (positive) functional $L$ on the unital commutative complex $*$-algebra ${\sA}_\C$.

Let $c\in C$. If   $x\in {\sA}_\C$, then $x^*xc\in C$ by Lemma 2.17(ii), so $L_0 (x^*xc)\geq 0$, and  
\begin{align}\label{pivarphicpositive}
    \langle \pi_L(c)\pi_L(x)1,\pi_L(x)1\rangle=L(x^*xc)=L_0 (x^*xc)\geq 0
\end{align}
by (\ref{gnslfunctional}). This shows that the operator $\pi_L(c)$ is nonnegative.

For $b\in {\sA}_\C$, the operator $\pi_L(b)$ is bounded by Proposition \ref{archmiboundedop}. Let\, $\ov{\pi_L(b)}$\, denote its continuous extension   to the  Hilbert space\, $\cH_L$. These operators form   a unital commutative  $*$-algebra of bounded operators. Its completion $\cB$ is a unital commutative $C^*$-algebra.

Let $\chi$ be a character  of $\cB$.  Then\,
$ \tilde{\chi}(\cdot):=\chi(\, \ov{\pi_L(\cdot)}\,)$ is a character of ${\sA}$.  If $c\in C$, then\, $\pi_L(c)\geq 0$\, by (\ref{pivarphicpositive})  and so\, $\ov{\pi_L(c)}\geq 0$. Hence  $\tilde{\chi}$ is $C$-positive, that is, $\tilde{\chi}\in \cK(C)$.   Therefore, $\tilde{\chi}(a)=\chi(\ov{\pi_L(a)}\,)>0$ by  $(i)_C$. Thus, if we realize  $\cB$ as a $C^*$-algebra of continuous functions on a compact Hausdorff space, the function corresponding to\, $\ov{\pi_L(a_0)}$\, is positive, so it has a positive  minimum $\delta$. Then\, $\ov{\pi_L(a_0)}\,\geq  \delta \cdot I$\, and hence
\begin{align*}0 <\delta =\delta L(1)=\langle\delta 1,1\rangle\leq \langle \pi_L (a)1,1\rangle =L(a_0)=L_0(a)\leq 0,
\end{align*}
which is the desired contradiction. $\hfill$ $\Box$
\smallskip

{\it  Proof of Theorem \ref{archmedps}(ii):}

We extend $L$ to a   $\C$-linear functional, denoted again by $L$,  on $\C_d[\ux]$ and consider the GNS representation $\pi_L$. By Proposition \ref{archmiboundedop}, the symmetric operators $\pi_L(x_1),\dotsc, \pi_L(x_d)$ are bounded. Hence their continuous extensions to the whole Hilbert space $\cH_L$  are pairwise commuting bounded self-adjoint operators $A_1,\dotsc,A_d$. Therefore, by Theorem \ref{mp-spec}, if $E$ denotes the spectral measure of this $d$-tuple $(A_1,\dotsc,A_d)$, then   $\mu(\cdot)= \langle
E(\cdot)1,1 \rangle_{\cH_L}$\, is a solution of the moment problem for $L$. 

Since the operators $A_j$ are bounded, the spectral measure $E$, hence $\mu$, has compact support. (In fact,\, ${\rm supp}~ E\subseteq  [-\|A_1\|,\|A_1\|] \times \dots \times[-\|A_d\|,\|A_d\|]$.) 
Hence, since  $L$ is $C({\sf f})$-positive by assumption,  Proposition  \ref{cqfimpliessuppckf} implies 
that ${\rm supp}\, \mu \subseteq \cK({\sf f})$. This shows that $L$ is a  $\cK({\sf f})$-moment functional. $\hfill$   $\Box$
\smallskip

The preceding proof of Theorem \ref{archmedps}(ii)  based on the spectral theorem is  probably  the most elegant approach to the moment problem for Archimedean quadratic modules. Next we derive Theorem \ref{archmedps}(i) from Theorem \ref{archmedps}(ii).
\smallskip

\noindent {\it  Proof of Theorem \ref{archmedps}(i):}

We argue in the same manner as in the second proof of Theorem \ref{schmps}  in Section \ref{momentproblemstrictpos}.
Assume to the contrary that $h\notin Q({\sf f})$. Since $Q({\sf f})$ is Archimedean,  Proposition \ref{eidelheit} and  Theorem \ref{archmedps}(ii)  apply to $Q({\sf f})$. By these results, there is a $Q({\sf f})$-positive linear functional $L$ on\,  $\R_d[\ux]$ satisfying $L(1)=1$ and $L(h)\leq 0$, and  this functional is a $\cK({\sf f})$-moment functional. 
Then  there  is a measure $\mu \in M_+(\R^d)$ supported on $\cK({\sf f})$ such that $L(p)=\int p\, d\mu$  for $p\in \R_d[\ux]$. (Note that $\cK({\sf f})$ is compact by  Corollary \ref{archicompact}.) Again  $h(x)>0$ on $\cK({\sf f})$,  $L(1)=1$, and $L(h)\leq 0$ lead to a contradiction. $\hfill$ $\Box$

\section{The moment problem for semi-algebraic sets contained in compact polyhedra}\label{polyhedron}

Let $k\in \N$. Suppose that  ${\sf f}=\{f_1,\dotsc,f_k\}$ is a set of linear polynomials of $\R_d[\ux]$. By a linear polynomial we  mean a polynomial of degree at most one. 
The semi-algebraic set  $\cK({\sf f})$ defined  by the linear polynomials $f_1,\dotsc,f_k$  is called a {\it polyhedron}.

Recall that $S({\sf f})$\index[sym]{PCfB@$ S({\sf f})$}  is the semiring of $\R_d[\ux]$  generated by $f_1,\dotsc,f_k$, that is, $S({\sf f})$ consists of all finite sums of terms $\alpha\, f_1^{n_1}\cdots f_k^{n_k}$,  where $\alpha\geq 0$ and $n_1,\dotsc,n_k\in \N_0$.

Further,  let ${\sf g}=\{g_0=1, g_1,\dotsc,g_r\}$, where $r\in \N_0$, be a finite subset of $\R_d[\ux]$. Recall that $C({\sf f},{\sf g}):= g_0\, S({\sf f})+ g_1 S({\sf f})+\cdots+ g_r S({\sf f})$  denotes the $S({\sf f})$-module considered in Example \ref{exacfg}, see (\ref{cfg}).

The following  lemma goes back to H. Minkowski. In the optimization literature it is  called {\it Farkas' lemma}.\index{Farkas lemma}
We will use it in the proof of Theorem \ref{prestel} below. 
\begin{lemma}\label{minkow} Let $h, f_1,\dotsc,f_k$ be linear polynomials of $\R_d[\ux]$ such that the set $\cK({\sf f})$ is not empty. If $h(x)\geq 0$ on $\cK({\sf f})$,    there exist numbers  $\lambda_0\geq 0,\dotsc,\lambda_m\geq 0$ such that $h=\lambda_0+\lambda_1f_1+\dots+\lambda_m f_m.$
\end{lemma}
\begin{proof}
    Let $E$ be the vector space spanned by the  polynomials $1,x_1,\dotsc,x_d$ and $C$ the cone in $E$ generated  by $1,f_1,\dotsc,f_m$. 
    It is easily shown that $C$ is closed in  $E$. 
    
    We have to prove that $h\in C$. Assume to the contrary that $g\notin C$. Then, by the separation of convex sets (Theorem A.26(ii)),  there exists a $C$-positive linear functional $L$ on $E$ such that $L(h)<0$. In particular, $L(1)\geq 0$, because $1\in C$. 
    
    Without loss of generality we can assume that $L(1)>0$. Indeed, if $L(1)=0$, we take a  point $x_0$ of the non-empty (!) set $\cK(\,{\sf \hat{f}}\,)$  and replace $L$ by $L^\prime=L+\varepsilon l_{x_0}$, where $l_{x_0}$ denotes the point evaluation at $x_0$ on $E$. Then $L^\prime$ is $C$-positive as well and $L^\prime(h)<0$ for  small $\varepsilon >0$. 
    
    Define a point\, $x:=L(1)^{-1}(L(x_1),\dotsc,L(x_d))\in \R^d$. Then $L(1)^{-1}L$  is the evaluation $l_x$ at the point $x$ for the polynomials $x_1,\dotsc,x_d$ and  for $1$, hence on the whole vector space $E$. Therefore,  $f_j(x)=l_x(f_j)=L(1)^{-1}L(f_j)\geq 0$ for all $j$, so that $x\in \cK(\, {\sf \hat{f}}\, )$, and $g(x)=l_x(h)=L(1)^{-1}L(h)<0$. This  contradicts the assumption. 
    \end{proof}

\begin{theorem}\label{prestel}
    Let $k\in \N$, $r\in \N_0$. Let\,  ${\sf f}=\{f_1,\dotsc,f_k\}$ and ${\sf g}=\{g_0=1, g_1,\dotsc,g_r\}$ be subsets of $\R_d[\ux]$\, such that the polynomials $f_1,\dotsc,f_k$ are linear.\, Suppose that the polyhedron\, $\cK(\, {\sf f}\, )$\, is  compact and  nonempty.
    %and that $\cK({\sf g})\subseteq \cK(\, {\sf f}\, )$. 
    \begin{itemize}
        \item[\em (i)]\, If $h\in \R_d[\ux]$ satisfies $h(x)> 0$ for all $x\in \cK({\sf g})$, then $h\in C({\sf f}, {\sf g})$, that is, 
        $h$ is a finite sum of polynomials
        \begin{align}\label{hsumof}
            \alpha g_j~  f_1^{n_1}\cdots f_k^{n_k}, ~~\textup{where}~~ \alpha\geq 0,~ j=1,\dotsc,r;~ n_1\dots,n_r\in \N_0.
        \end{align}
        \item[\em (ii)]\,\, \,  A linear functional $L$  on $\R_d[\ux]$  is a $\cK( {\sf  f} )\cap\cK({\sf g})$--moment functional if and only if 
        \begin{align}\label{solvconsemiring}
            L( g_j\, f_1^{n_1}\cdots f_k^{n_k})\geq 0\,  \quad \text{for~ all}~~~ j=0,\dotsc, r; n_1,\dotsc,n_k\in \N_0.
        \end{align}
    \end{itemize}
\end{theorem} 
\begin{proof}
    First we show that the semiring\, $S( {\sf f} )$ is Archimedean. Let $j\in \{1,\dotsc,d\}$. Since the set $\cK(\, {\sf f}\, )$ is compact, there exists a $\lambda>0$ such that $\lambda\pm x_j> 0$ on  $\cK(\, {\sf  f}\, )$. Hence, since $\cK(\, {\sf f}\, )$ is nonempty, Lemma \ref{minkow} implies that $(\lambda\pm x_j)\in S({\sf f})$. Hence  $S( {\sf f})$ is Archimedean by Lemma \ref{boundedele2}(ii).  
    
    The only if part in (ii) is obvious. Since $S( {\sf f})$ is Archimedean,  Theorem \ref{archmedps} applies to the $S( {\sf f})$-module  $ C({\sf f}, {\sf g})$ and gives the other assertions. Note that the requirements (\ref{solvconsemiring}) suffice, since $h$ in (i) is a sum of terms (\ref{hsumof}).
    \end{proof}

We state the special case $r=0$ of a polyhedron $\cK(\, {\sf f})$ separately as a corollary. Assertion (i) is called {\it Handelman's theorem}.

\begin{corollary}\label{prestelcor}
    Let $k\in \N$. Suppose that\,  ${\sf f}=\{f_1,\dotsc,f_k\}$ is a set of linear polynomials of $\R_d[\ux]$ such that the polyhedron\, $\cK(\, {\sf f}\, )$\, is  compact and  nonempty.
    %and that $\cK({\sf g})\subseteq \cK(\, {\sf f}\, )$. 
    \begin{itemize}
        \item[\em (i)]\, If $h\in \R_d[\ux]$ satisfies $h(x)> 0$ for all $x\in \cK({\sf f})$, then $h\in S({\sf f}).$
        \item[\em (ii)]\,\, \,  A linear functional $L$  on $\R_d[\ux]$  is a $\cK( {\sf  f})$--moment functional if and only if 
        \begin{align}\label{solvconsemiringcor}
            L( f_1^{n_1}\cdots f_k^{n_k})\geq 0\,  \quad \text{for~ all}~~~ n_1,\dotsc,n_k\in \N_0.
        \end{align}
    \end{itemize}
\end{corollary} 
\begin{proof} Set $r=0, g_0=1$ in Theorem \ref{prestel} and note that $\cK(C({\sf f},{\sf g}))= \cK({\sf f}).$
    \end{proof}
\section{Examples and applications}\label{examplesmp}

Throughout this section, 
${\sf f}=\{f_1,\dotsc,f_k\}$ is a finite subset of $\R_d[\ux]$
and $L$ denotes a linear functional on $\R_d[\ux].$

If $L$ is a  $\cK({\sf f})$-moment functional,  it is obviously $T({\sf f})$-positive,  $Q({\sf f})$-positive, and $S({\sf f})$-positive. 
Theorems\, \ref{mpschm},  \ref{archmedps}(ii), and \ref{prestel}(ii) deal with the converse implication and are the  main solvability criteria for the moment problem  in  this chapter. 

First we discuss Theorems \ref{mpschm} and  \ref{archmedps}(ii).
Theorem \ref{mpschm} 
applies to {\it each} compact semi-algebraic set $\cK({\sf f})$ and implies that $L$ is a $\cK({\sf f})$-moment functional if and only if it is   $T({\sf f})$-positive. For Theorem \ref{archmedps}(ii) the compactness of the set $\cK({\sf f})$ is not 
sufficient; it requires that the quadratic module $Q({\sf f})$ is Archimedean. In this case,  $L$ is a $\cK({\sf f})$-moment functional if and only if it is $Q({\sf f})$-positive.
\begin{example}
    Let us begin with a single  polynomial $f\in \R_d[\ux]$ for which  the set  $\cK(f)=\{x\in \R^d:f(x)\geq 0\}$ is compact. (A simple example is the $d$-ellipsoid given by $f(x)=1-a_1x_1^2-\dots -a_dx_d^2$, where $a_1>0,\dotsc,a_d>0$.) Clearly,   $T(f)=Q(f)$. Then, {\it $L$ is a  $\cK(f)$-moment functional if and only if it is $T(f)$-positive, or equivalently, if $L$ and $L_f$ are positive functionals on $\R_d[\ux]$.}

    Now we add further polynomials  $f_2,\dotsc,f_k$ and set ${\sf f}=\{f,f_2,\dotsc,f_k\}$. (For instance, one may take  coordinate functions as $f_j=x_l$.) Since $T(f)$ is Archimedean (by Proposition \ref{prearchcom}, because  $\cK(f)$ is compact), so is the quadratic module $Q({\sf f})$. Therefore, {\it $L$ is a $\cK({\sf f})$-moment functional if and only if  it is $Q(f)$-positive, or equivalently, if  $L, L_f,L_{f_2}, \dots, L_{f_k}$  are positive  functionals on $\R_d[\ux]$.}  $\hfill \circ$
\end{example}

\begin{example}\label{n-dimintervalexist} {\it ($d$-dimensional compact interval~ $[a_1,b_1]\times\dots \times [a_d,b_d]$)}\index{Moment problem! for $d$-dimensional compact intervals}\\
    Let  $a_j,b_j\in \R$, $a_j< b_j,$ and set $f_{2j-1}:=b_j-x_j$, $f_{2j}:=x_j-a_j,$ for $j=1,\dotsc,d$. Then
    the semi-algebraic set $\cK({\sf f})$ for \, ${\sf f}:=\{f_1,\dotsc,f_{2d}\}$ is the $d$-dimensional interval $[a_1,b_1]\times\dots \times [a_d,b_d]$. 
    
    Put $\lambda_j=|a_j|+|b_j|.$ Then  $\lambda_j-x_j=f_{2j-1}+ \lambda_j-b_j$  
    and $\lambda_j+x_j=f_{2j}+\lambda_j+a_j$ are $Q({\sf f})$, so each
    $x_j$  is a bounded element with respect to the quadratic module $Q({\sf f})$. Hence 
    $Q({\sf f})$ is Archimedean by Lemma \ref{boundedele2}(ii).
    
    Thus, {\it  $L$  is a  $\cK({\sf f})$-moment functional if and only if it is $Q(f)$-positive, or equivalently, if\,  $ L_{f_1},L_{f_2}, \dots, L_{f_k}$  are positive  functionals, that is,}
    \begin{align}\label{solvn-diminter}
        L((b_j{-}x_j)p^2)\geq 0~~ \text{\textit{and}} ~~ L((x_j{-}a_j)p^2)\geq 0~~ \text{\textit{for}} ~~ j=1,\dotsc,d,\, p\in \R_d[\ux].
    \end{align}
    Clearly, (\ref{solvn-diminter})  implies that\, $L$\, itself is positive, since  $L=(b_1{-}a_1)^{-1}(L_{f_1}{+}L_{f_2})$. $\hfill \circ$
\end{example}

\begin{example} {\it ($1$-dimensional interval $[a,b]$)}\\
    Let $a<b$, $a,b\in \R$ and let $l,n\in \N$ be odd. We set  $f(x):=(b-x)^l(x-a)^n$. Then  $\cK( f)=[a,b]$\, and\,  $T( f)=\sum \R[x]^2+f\sum \R[x]^2$. Hence, by Theorem \ref{mpschm}, {\it a linear functional $L$ on $\R[x]$ is an $[a,b]$-moment functional if and only if $L$ and $L_f$  are positive functionals on $\R[x]$}. 
    
    This result extends Hausdorff's Theorem 3.13. 
    It should be noted that this solvability criterion holds for arbitrary (!) odd numbers $l$ and $n$, while the equality ${\Pos}([a,b])=T(f)$ is  only true if $l=n=1$, see  Exercise 3.4 b. in Chapter 3. $\hfill \circ$
\end{example}

\begin{example} ({\it Simplex  in $\R^d, d\geq 2$})\\
    Let $f_1=x_1,\dotsc,f_d=x_d, f_{d+1}=1- \sum_{i=1}^d x_i, k=d+1$. Clearly, $\cK({\sf f})$ is the simplex
    \begin{align*}
        K_d=\{ x\in \R^d: x_1\geq 0,\dotsc, x_d\geq 0,\, x_1+\dots+x_d\leq 1\, \}.
    \end{align*}
    Note that $1-x_j=f_{d+1}+\sum_{i\neq j} f_i$ and $1+x_j=1+f_j$. Therefore,
    $1\pm x_j\in Q({\sf f})$ and $1\pm x_j\in S({\sf f})$. Hence, by Lemma \ref{boundedele2}(ii),  
    the quadratic module $Q({\sf f})$ and the semiring  $S({\sf f})$ are Archimedean. Therefore, Theorem \ref{archmedps} applies to $Q({\sf f})$ and  Theorem \ref{prestel}  applies to $S({\sf f})$.
    We restate only the  results on the moment problem.  
    
    By Theorems \ref{archmedps}(ii) and \ref{prestel}(ii), 
    {\it $L$ is a\, $K_d$--moment functional if and only if}
    \begin{align*}
        L(x_ip^2)\geq 0, ~ i=1,\cdots,d,~~\text{\textit{and}}~~~ L((1- ( x_1+x_2+\dots+x_d))p^2)\geq 0 ~~~\text{\textit{for}}~~p \in \R_d[\ux],  
    \end{align*}
    {\it or equivalently,}
    \begin{align*}
        \hspace{0,5cm}L(x_1^{n_1}\dots x_d^{n_d}(1- (x_1+\dots+x_d))^{n_{d+1}}) \geq 0 \quad \text{\textit{for}}~~~n_1,\dotsc,n_{d+1}\in \N_0.\hspace{0,5cm}\Box~~\circ
    \end{align*}
\end{example}

\begin{example}\label{Delta_dsimplex} ({\it Standard simplex $\Delta_d$ in $\R^d$})\\
    Let \, $f_1=x_1,\dotsc,\, f_d=x_d,\, f_{d+1}=1- \sum_{i=1}^d x_i,\,  f_{d+2}=-f_{d+1},\, k=d+2.$ Then the semi-algebraic set $\cK({\sf f})$ is the standard simplex
    \begin{align*}
        \Delta_d=\{x\in \R^d: x_1\geq 0,\dotsc,x_d\geq 0, x_1+\dots+x_d=1\}.
    \end{align*}
    Let $S_0$ denote the polynomials of $\R_d[\ux]$ with nonnegative coefficients and $\cI$ the ideal generated by $1-(x_1+\dots +x_d)$. 
    Then $S:=S_0+\cI$ is a semiring of $\R_d[\ux]$. Since $1\pm x_j\in S$, $S$ is Archimedean. The characters of $\R_d[\ux]$ are the evaluations at points of $ \R^d$. Obviously, $x\in \R^d$ gives a $S$-positive character if and only if $x\in \Delta_d$.  
    
    Let $f\in \R_d[\ux]$ be such that   $f(x)>0$ on $\Delta_d$. Then, $f\in S$ by  Theorem \ref{archpos},  so 
    \begin{align}\label{polyaformoff}
        f(x)=g(x)+ h(x)(1-(x_1+\dots + x_d)),\quad {\rm where}~~~ g\in S_0, ~h\in \R_d[\ux].
    \end{align}
    From Theorem \ref{prestel}(ii) it follows  that
    {\it $L$ is a $\Delta_d$-moment functional if and only if}
    \begin{align*}
        L(x_1^{n_1}\dots x_d^{n_d})\geq 0,~  L(x_1^{n_1}\dots x_d^{n_d}(1{-}(x_1{+}\dots{+}x_d))^r)=0, ~~ n_1,\dotsc,n_d\in \N_0, r\in \N.\circ
    \end{align*}
\end{example}

From the preceding example it is only a small step to derive an elegant proof of the following classical  {\it theorem of G. Polya}.\index{Theorem! Polya}
\begin{proposition}\label{proofPolya}
    Suppose that $f\in \R_d[\ux]$ is a homogeneous polynomial such that $f(x)>0$ for all\,  $x\in \R^d\backslash \{0\}$, $ x_1\geq 0, \dots,x_d\geq 0$.
    Then there exists an $n\in \N$ such that all  coefficients of the polynomial $(x_1+\dots+x_d)^n f(x)$ are nonnegative.
\end{proposition}
\begin{proof}
    We use Example \ref{Delta_dsimplex}. As noted therein,  Theorem \ref{archpos} implies that $f$ is of the form (\ref{polyaformoff}). We replace in (\ref{polyaformoff}) each variable $x_j, j=1,\dotsc,d,$ by $x_j(\, \sum_{i=1}^d x_i)^{-1}$. Since $(1-\sum_j x_j (\sum_i x_i)^{-1}) =1- 1=0,$  the second summand in (\ref{polyaformoff}) vanishes after this substitution.  Hence, because $f$ is homogeneous,   (\ref{polyaformoff}) yields
    \begin{align}\label{polyaidy}
        \big(\, \sum\nolimits_i x_i\big)^{-m}f(x)=g\big(x_1 \big( \, \sum\nolimits_i x_i \big)^{-1},\dotsc,x_d\big(\, \sum\nolimits_i x_i\big)^{-1}\big),
    \end{align}
    where   $m=\deg(f)$. Since $g\in S_0$, $g(x)$ has only nonnegative coefficients. Therefore, after multiplying (\ref{polyaidy})  by  $(\sum_i x_i)^{n+m}$ with $n$ sufficiently large to clear the denominators,  we obtain the assertion.
\end{proof}

Finally, we mention  two examples of polyhedrons  based on Corollary \ref{prestelcor}(ii). 
\begin{example} $[-1,1]^d$\\
    Let $k=m=2d$ and $f_1=1-x_1, f_2=1+x_1,\dotsc,f_{2d-1}=1-x_d,f_{2d}=1+x_d.$ Then $ \cK(\,{\sf f}\,)=[-1,1]^d$. Therefore, by  Corollary \ref{prestelcor}(ii), {\it a linear functional $L$ on $\R_d[x_d]$ is a $[-1,1]^d$-moment functional if and only if}
    \begin{align*}
        L( (1-x_1)^{n_1}(1+x_1)^{n_2}\dotsm (1-x_d)^{n_{2d-1}}(1+x_d)^{n_{2d}})&\geq 0&
        \text{for }n_1,\dotsc,n_{2d}\in \N_0. \hspace{0,2cm} \circ
    \end{align*}
\end{example}
\begin{example} ({\it Multidimensional\, Hausdorff\, moment\, problem on $[0,1]^d$})\index{Multidimensional Hausdorff moment problem}\\
    Set  $f_1=x_1, f_2=1-x_1,\dotsc,f_{2d-1}=x_d,f_{2d}=1-x_d, k=2d$. Then $\cK(\,{\sf f}\,)= [0,1]^d$. Let $s=(s_\gn)_{\gn\in \N_0^d}$ be a 
    multisequence. We define  the shift $E_j$ of the $j$-th index by  \begin{align*}
        (E_js)_\gm=s_{(m_1,\dotsc,m_{j-1},m_j+1,m_{j+1},\dotsc,m_d)}, ~~~\gm \in \N_0^d.
    \end{align*}
\end{example}

    \begin{proposition}\label{multihausdorff}
        The following five statements are equivalent:
        \begin{itemize}
            \item[\em (i)]~ $s$ is a Hausdorff moment sequence on  $[0,1]^d$.
            \item[\em (ii)]~ $L_s$ is a $[-1,1]^d$-moment functional on $\R_d[\ux].$ 
            \item[\em (iii)]~ $L_s( x_1^{m_1}(1-x_1)^{n_1}\cdots x_d^{m_d}(1-x_d)^{n_d})\geq 0$~ for all~ $\gn,\gm\in \N_0^d$.
            \item[\em (iv)]~ $((I-E_1)^{n_1}\dots (I-E_d)^{n_d}s)_\gm \geq 0$~ for all~ $\gn,\gm\in \N_0^d$.
            \item[\em (v)]~~ \begin{align*}\sum_{\gj\in \N_0^d, \gj\leq \gn}~ (-1)^{|\gj|}\binom{n_1}{ j_1}\cdots \binom{n_d}{ j_d}\,  s_{\gm+\gj}\geq 0\end{align*} for all $\gn,\gm\in \N_0^d$. Here $|\gj|:=j_1+\dots+j_d$ and $\gj\leq \gn$ means that $j_i\leq n_i$ for $i=1,\dotsc,d$.
        \end{itemize} 
    \end{proposition}
    \begin{proof}
        (i)$\leftrightarrow$(ii) holds by definition. Corollary \ref{prestelcor}(ii) yields (ii)$\leftrightarrow$(iii). 
        Let $\gn, \gm\in \N_0^d$. 
        We repeat  the computation from the proof of Theorem 3.15 and derive
        \begin{align*}
            L_s( x_1^{m_1}(1-x_1)^{n_1}&\cdots x_d^{m_d}(1-x_d)^{n_d})= ((I-E_1)^{n_1}\dots (I-E_d)^{n_d}s)_\gm \\ &=\sum_{\gj\in \N_0^d, \gj\leq \gn}~ (-1)^{|\gj|}\binom{n_1}{ j_1}\cdots \binom{n_d}{ j_d}  s_{\gm+\gj}.
        \end{align*}
        This identity implies the equivalence of conditions (iii)--(v).
    \end{proof}\hfill$\circ$ 

\section{Exercises}
\begin{itemize}
    \item[1.]\, Suppose that $Q$ is a quadratic module of a commutative real algebra ${\sA}$. Show that $Q\cap(-Q)$ is an ideal of ${\sA}$.
    This  ideal is called the {\it support ideal} of $Q$.
    \item[2.]\, Let $K$ be a closed subset of $\R^d$.  Show that ${\Pos}(K)$ is  saturated. 
    \item[3.]\, 
    Formulate  solvability criteria in terms of localized functionals and in terms of $d$-sequences for the following sets.
    \begin{itemize}
        \item[a.]\, Unit ball of $\R^d$.
        \item[b.]\, $\{x\in \R^d: x_1^2+\dots+x_d^2\leq r^2,~ x_1\geq 0,\dotsc, x_d\geq 0\}$.
        \item[c.]\, $\{(x_1,x_2,x_3,x_4)\in \R^4: x_1^2+x_2^2\leq 1,x_3^2+x_4^2\leq 1\}$.
        \item[d.]\, $\{(x_1,x_2,x_3)\in \R^3: x_1^2+x_2^2+x_3^2\leq 1, x_1+x_2+x_3\leq 1\}.$
        \item[e.]\, $ \{x\in \R^{2d}:x_1^2+x_2^2=1,\dotsc, x_{2d-1}^2+x_{2d}^2=1\}$.
    \end{itemize}
    \item[4.]\, Decide whether or not the following quadratic modules $Q({\sf f})$ are Archimedean.
    \begin{itemize}
        \item[a.]\, $f_1=x_1,f_2=x_2, f_3=1-x_1x_2, f_4=4-x_1x_2$.
        \item[b.]\, $f_1=x_1,f_2=x_2, f_3=1-x_1-x_2.$
        \item[c.]\, $f_1=x_1,f_2=x_2, f_3=1-x_1x_2$.
    \end{itemize}
    \item[5.]\, Let $f_1,\dotsc,f_k,g_1,\dotsc,g_l\in \R_d[\ux]$. Set   ${\sf {g}}=(f_1,\dotsc,f_k,g_1,\dotsc,g_l)$,  ${\sf {f}}=(f_1,\dotsc,f_k)$.
    Suppose that $Q({\sf f})$ is Archimedean. Show that each $Q({\sf g})$-positive linear functional $L$ is a determinate $\cK({\sf g})$-moment functional. 
    \item[6.]\, Formulate solvability criteria for the moment problem of the following semi-algebraic sets $\cK({\sf f})$.
    \begin{itemize}
        \item[a.]\, $f_1=x_1^2+\dots+x_d^2, f_2=x_1,\dotsc, f_k=x_{k-1}$, where $2\leq k\leq d+1$.
        \item[b.]\, $f_1=x_1, f_2=2-x_1,f_3=x_2, f_4=2-x_2, f_5=x_1^2-x_2,$ where $d=2$.
        \item[c.]\, $f_1=x_1^2+x_2^2, f_2=ax_1+bx_2, f_3=x_2,$ where $d=2, a,b\in \R$.
    \end{itemize}
    \item[7.] \, Let $d=2$, $f_1=1-x_1, f_2=1+x_1, f_3=1-x_2,f_4=1+x_2, f_5=1-x_1^2-x_2^2$~ and ${\sf f}=(f_1,f_2,f_3,f_4,f_5).$ Describe the set $\cK(\, {\sf f}\, )$ and use Theorem \ref{prestel}(ii) to characterize  $\cK(\, {\sf f}\, )$-moment functionals.
    \item[8.]\, Find a $d$-dimensional version of Exercise 7, where $d\geq 3.$
    \item[9.]\, ({\it Tensor product of preorderings})\\ Let $n,k\in \N$. Suppose that ${\sf f}_1$ and  ${\sf f}_2$ are finite subsets  of $\R_n[\ux]\equiv \R[x_1,\dotsc,x_n]$ and $\R_k[\ux']\equiv \R[x_{n+1},\dotsc,x_{n+k}]$, respectively, such that the semi-algebraic sets $\cK({\sf f}_1)$ of $\R^n$ and $\cK({\sf f}_2)$ of $\R^k$ are compact. Define a subset $T$ of $\R[x_1,\dotsc,x_{n+k}]$ by
    \begin{align*}
        T:= \Big\{ p(x,x')=\sum_{j=1}^r p_j(x)q_j(x'): ~p_1,\dotsc,p_r\in T({\sf f}_1),\,  q_1,\dotsc,q_r\in T({\sf f}_2),\, r\in \N\Big\}.
    \end{align*}
    \begin{itemize}
        \item[a.] Show that $T$ is an Archimedean semiring of $\R[x_1,\dotsc,x_{n+k}]$.
        \item[b.] Give an example of ${\sf f}_1$ and  ${\sf f}_2$  for which $T$ is not a preordering.
        \item[c.] Let $p\in \R[x_1,\dotsc,x_{n+k}]$. Suppose  $p(x,x')>0$ for all $x\in \cK({\sf f}_1)$, $x'\in\cK({\sf f}_2)$. Prove that  $p\in T$.
    \end{itemize} 
    
    Hint: The preorderings $ T({\sf f}_1)$ and $T({\sf f}_2)$ are 
    Archimedean (Proposition \ref{prearchcom}). Hence $f\otimes 1$ and $1\otimes g$  satisfy the Archimedean condition for $f\in T({\sf f}_1)$ and $g\in T({\sf f}_2)$. The semiring $T$ is generated by these elements, so $T$ is Archimedean. For b.) try $p=(x_1-x_{n+1})^2$. For c.), apply the Archimedean Positivstellensatz.
    \item[10.]~ ({\it Supporting polynomials of compact convex sets of $\R^d$})\\
    Let   $K$ be a non-empty compact convex subset of $\R^d$. By a {\it supporting polynomial} of $K$ at some point $t_0\in K$ we mean a  polynomial  $h\in \R_d[\ux]$ of degree one such that $h(t_0)=0$ and $h(t)\geq 0$ for all $t\in K$. (In this case, $t_0$ a is a boundary point of $K$.) Suppose that $H$ is a set of supporting polynomials at points of $K$ such that
    $$
    K=\{ t\in \R^d: h(t)\geq 0~~\textup{ for all}~~  h\in H \}.$$                                                                                         
    \begin{itemize}
        \item[a.] Prove that the semiring $S(H)$ of $\R_d[\ux]$ generated by $H$ is Archimedean.
        \item[b.] Let $f\in \R_d[\ux]$ be such that $f(t)>0$ for all $t\in K$.  Prove that $f\in S(H).$
    \end{itemize}
    \item[11.]~Elaborate Exercise 10. for the unit disc $K=\{(x,y)\in \R^2: x^2+y^2\leq 1\}$ and $H:=\{ h_\theta:=1+x\, \cos(\theta)+y\, \sin\theta: \theta\in [0,2\pi)\}$ or for appropriate subsets of $K$.
    \item[12.]~({\it Reznick's theorem} [Re2])\index{Theorem! Reznick}\\
    Let $f\in \R_d[\ux]$ be a homogeneous polynomial such that $f(x)> 0$ for $x\in \R^d$, $x\neq 0$. Prove that there exists an $n\in \N$ such that $(x_1^2+\dots+x_d^2)^nf(x)\in \sum \R_d[\ux]^2$.\smallskip
    
    Hint: Mimic the proof of Proposition \ref{proofPolya}:~ Let   $T$ denote the preordering $\sum \R_d[\ux]+\cI$, where  $\cI$ is the ideal generated by the polynomial\, $1-(x_1^2+\dots+ x_d^2)$. Show that  $T$-positive characters corresponds to  points of the unit sphere, substitute $x_j(\sum_i x_i^2)^{-1}$ for $x_j$, apply Theorem \ref{prestel}(i) to $T$, and clear denominators. 
\end{itemize}

\section{Notes}
The interplay between real algebraic geometry  and the moment problem for compact semi-algebraic sets and the corresponding Theorems  \ref{schmps} and \ref{mpschm} were discovered by the author in [Sm6]. 
A small gap in the proof of [Sm6, Corollary 3] (observed by A. Prestel) 
was immediately repaired by  the reasoning of the above proof of Proposition \ref{prearchcom} (taken from [Sm8, Proposition 18]).

The fact that the preordering  is Archimedean in the compact case was first noted  by 
T. W\"ormann [W\"o]. An algorithmic proof of Theorem \ref{schmps} was developed by M. Schweighofer [Sw1], [Sw2]. 

The operator-theoretic proof of Theorem \ref{archmedps}(ii) given above 
is long known among operator theorists; it was  used in [Sm6]. The operator-theoretic approach to the multidimensional moment theory was investigated in detail by F. Vasilescu [Vs1], [Vs2].

The Archimedean Positivstellensatz  (Theorem \ref{archpos}) has a long history. It was proved in various versions  by M.H. Stone [Stn], R.V. Kadison [Kd], J.-L. Krivine [Kv1], E. Becker and N. Schwartz [BS], M. Putinar [Pu2], and T. Jacobi [Jc]. The general version for quadratic modules is due to Jacobi [Jc], while the version  for semirings was proved much earlier by  Krivine [Kr1]. A more general version and a detailed discussion can be found in [Ms1, Section 5.4]. The unified approach to Theorem \ref{archpos} in  Section \ref{reparchimodiules} using the dagger cones is based on results obtained in the paper [SmS23]. Theorem \ref{auxsemiring} and Example \ref{module} are also taken from [SmS23].

M. Putinar  [Pu2] has proved that a finitely generated quadratic module $Q$ in $\R_d[\ux]$ is Archimedean if (and only if) there exists a polynomial $f\in Q$ such that the set $\{x\in \R^d: f(x)\geq 0\}$ is compact.

Corollary  \ref{mpwithboundeddensity} and its non-compact version  in Exercise 14.11 below are  from [Ls3]. The moment problem with bounded densities is usually called the {\it Markov moment problem} or $L$-moment problem. In dimension one it goes back to A.A. Markov  [Mv1], [Mv2], see [AK], [Kr2]. An interesting more recent work is [DF]. The multidimensional case  was studied in  [Pu1], [Pu3], [Pu5], [Ls3], [Ls4].

For  compact polyhedra with nonempty interiors 
Corollary \ref{prestelcor}(i) was proved by D. Handelman [Hn]. A  special case was treated  earlier by J.-L.  Krivine [Kv2]. A related version can be found in [Cs, Theorem 4]. The  general Theorem \ref{prestel} is taken from [SmS23]; it   is a slight generalization of [PD, Theorem 5.4.6]. 

Polya's theorem was proved in [P]. Polya's original proof is  elementary; the elegant proof given in the text is from [W\"o].   Proposition \ref{multihausdorff} is a classical result obtained in [HS]. It should be noted that Reznick's theorem [Re2] can be derived as an immediate consequence of Theorem \ref{schmps}, see [Sr3, 2.1.8].

Reconstructing the shape of subsets of $\R^d$ from its moments with respect to the Lebesgue measure is another interesting topic, see e.g. [GHPP] and [GLPR].

\end{document}